\documentclass[10pt]{amsart}
\usepackage[utf8]{inputenc}
\usepackage[T1]{fontenc}
\usepackage[english]{babel}
\title[Martin boundary of relatively hyperbolic groups]{The Martin boundary of relatively hyperbolic groups with virtually abelian parabolic subgroups}
\author[M. Dussaule]{Matthieu Dussaule}
\author[I. Gekhtman]{Ilya Gekhtman}
\author[V. Gerasimov]{Victor Gerasimov}
\author[L. Potyagailo]{Leonid Potyagailo}
\date{}
\usepackage{silence}
\usepackage{comment}
\usepackage{amsmath}
\usepackage{amssymb}
\usepackage{dsfont}
\usepackage{amsfonts}
\usepackage{amsthm}
\usepackage{tikz}
\usepackage{graphicx}
\usepackage{fancyhdr}
\usepackage{caption}
\usepackage{enumitem}
\usepackage{mathtools}
\usepackage[colorlinks=true,pdfstartview=FitH,bookmarks=false]{hyperref}
\hypersetup{urlcolor=black,linkcolor=black,citecolor=black,colorlinks=true}

\newcommand\N{\mathbb{N}}
\newcommand\Z{\mathbb{Z}}

\newcommand\R{\mathbb{R}}

\newcommand\Ss{\mathds{S}}

\theoremstyle{definition}
\newtheorem{assumption}{Assumption}
\newtheorem*{assumption*}{Assumption 2'}

\theoremstyle{plain}
\newtheorem{definition}{Definition}[section]
\newtheorem{proposition}[definition]{Proposition}
\newtheorem{corollary}[definition]{Corollary}
\newtheorem{theorem}[definition]{Theorem}
\newtheorem{lemma}[definition]{Lemma}
\newtheorem*{thm*}{Theorem}
\newtheorem*{prop*}{Proposition}
\newtheorem*{lem*}{Lemma}

\theoremstyle{remark}
\newtheorem{remark}{Remark}[section]

\newtheorem*{rem*}{Remark}

\begin{document}
\begin{abstract}
Given  a  probability  measure  on  a  finitely  generated  group,  its  Martin
boundary is a way to compactify the group using the Green's function
of the corresponding random walk.
We give a complete topological characterization of the Martin boundary of finitely supported random walks on relatively hyperbolic groups with virtually abelian parabolic subgroups. In particular, in the case of nonuniform lattices in the real hyperbolic space $\mathcal{H}^{n}$, we show that the Martin boundary coincides with the $CAT(0)$ boundary of the truncated space, and thus when $n=3$, is homeomorphic to the Sierpinski carpet.
\end{abstract}
\maketitle
\section{Introduction and Statement of results}\label{Sectionintroduction}

\subsection{Random walks on relatively hyperbolic groups}\label{SectionRandomwalksrelhypgroups}
A probability measure $\mu$ on a countable group $\Gamma$ determines a $\Gamma$-invariant Markov chain with transition probabilities $p(x,y)=\mu(x^{-1}y)$, called a random walk.

Connecting asymptotic properties of this random walk to the geometry of Cayley graphs of $\Gamma$ has been a fruitful line of research. One way to do this is through relating the Green's function of $\mu$ to some natural metric on $\Gamma$, and the probabilistically defined Martin boundary of $\mu$ to some geometric boundary of $\Gamma$.

The Green's function $G$ of $(\Gamma,\mu)$ is defined as
$$G(x,y)=\sum^{\infty}_{n=0}\mu^{*n}(x^{-1}y).$$ 
It describes the amount of time a random path starting 
at $x$ is expected to spend at $y$.
We now fix a base point $o$ in $\Gamma$.
For each $y\in \Gamma$ the function 
$K_{y}:\Gamma \to \mathbb{R}$ defined by $K_{y}(x)=G(x,y)/G(o,y)$ is called a Martin kernel.
The compactification of $\Gamma$ in the space of functions $\Gamma \to \mathbb{R}$ of the Martin kernels $K_{y}:y\in \Gamma$ is called the Martin compactification $\overline{\Gamma}_{\mu}$ and $\partial_{\mu}\Gamma=\overline{\Gamma}_{\mu} \setminus \Gamma$ is called the Martin boundary.
These definitions also make sense for more general measures $\mu$ (see Section~\ref{SectionMartinboundariesofrandomwalks}).

Giving a geometric description of the Martin boundary is often a difficult problem. 
Margulis showed that for centered finitely supported random walks on nilpotent groups, the Martin boundary always is trivial \cite{Margulis}.
On the other hand, for noncentered random walks with exponential moment on abelian
groups of rank $k$, Ney and Spitzer \cite{NeySpitzer}
showed that the Martin boundary is homeomorphic to a sphere of dimension $k-1$.

For hyperbolic groups with finitely supported measures, Ancona \cite{Ancona} proved that the Green's function is roughly multiplicative along word geodesics. He used this to identify the Martin boundary with the Gromov boundary. 

A finitely generated group $\Gamma$ is called hyperbolic relative to a system of subgroups $\Omega$ if $\Gamma$ admits a convergence action on a compact metric space $T$ such that every point of $T$ is either conical or bounded parabolic and the elements of $\Omega$ are the stabilizers of the parabolic points.
We will give more details in Section~\ref{Sectionrelativelyhyperbolicgroups}.

We will always assume that the action $\Gamma \curvearrowright T$ is minimal and nonelementary, that is $T$ contains more than two points. We again refer to Section~\ref{Sectionrelativelyhyperbolicgroups} for more details.

For relatively hyperbolic groups with finitely supported measures, Gekhtman, Gerasimov, Potyagailo and Yang proved a generalization of Ancona's multiplicativity estimate, see \cite[Theorem~5.2]{GGPY}.
Its most general formulation uses the Floyd distance, which is a rescaling of the word distance.
It will be defined in Section~\ref{SectionrelhypandFloyd}.

The authors of \cite{GGPY} used this estimate to show that the Martin boundary covers the Bowditch boundary -the compact set $T$ of the above definition, or equivalently, the boundary of a proper geodesic Gromov hyperbolic space $X$ on which $\Gamma$ acts geometrically finitely.  Moreover the preimage of any conical point is a singleton. Determining the Martin boundary is thus reduced to determining the preimage of parabolic points. 

Dussaule \cite{Dussaule} gave a geometric description of the Martin boundary of finitely supported random walks on free products of abelian groups, identifying it with the visual boundary of a $CAT(0)$ space on which the group acts cocompactly.

The key technical result of \cite{Dussaule} extends results of Ney and Spitzer \cite{NeySpitzer} to more general chains.
It states that the Martin boundary of non-centered -or strictly sub-Markov- chains on $\mathbb{Z}^{k}\times \{1,...,N\}$, $N\in \N$, is a sphere of dimension $k-1$.
We give a precise formulation below (see Proposition~\ref{latticeMartinboundary}).

In this paper, we extend these two results of \cite{GGPY} and \cite{Dussaule} to determine the Martin boundary of finitely supported measures on groups relatively hyperbolic with respect to virtually abelian subgroups. These include the following well-known classes of groups:
\begin{itemize}
\item geometrically finite Kleinian groups,
\item limit groups,
\item groups acting freely on $\mathbb{R}^n$ trees.
\end{itemize}

\medskip
If $\overline{\Gamma}$ is a compactification of $\Gamma$, define the corresponding boundary as $\partial \Gamma:=\overline{\Gamma}\setminus \Gamma$.
Call a boundary $\partial \Gamma$ of $\Gamma$ a $Z$-boundary if the following holds (see Section~\ref{Sectioncompactification} for the details): an unbounded sequence $(g_{n})$ of $\Gamma$ converges to a point in $\partial \Gamma$ if and only if either $(g_{n})$ converges to a conical point in the Bowditch boundary, or there is a coset $gP$ of a parabolic subgroup $P$ such that one (and thus any) closest point projections of $(g_{n})$ to $gP$ converges to a point in the $CAT(0)$ boundary of $gP$.

Our main result is the following.
\begin{theorem}\label{maintheorem}
Let $\Gamma$ be a  hyperbolic group
relative to a collection of infinite virtually abelian subgroups.
Let $\mu$ be a measure on $\Gamma$ whose finite support generates $\Gamma$ as a semigroup.
Then, the Martin boundary is a $Z$-boundary.
\end{theorem}

We also prove the following. 
\begin{theorem}\label{minimal}
Let $\Gamma$ be a  hyperbolic group
relative to a collection of infinite virtually abelian subgroups.
Let $\mu$ be a measure on $\Gamma$ whose finite support generates $\Gamma$ as a semigroup.
Then, every point of the Martin boundary is minimal.
\end{theorem}

We recover from Theorem~\ref{maintheorem}, as a particular case, that the Martin boundary of a non-virtually cyclic hyperbolic group is the Gromov boundary.
Since two $Z$-boundaries are equivariantly homeomorphic (see Lemma~\ref{equivalentZbdry} for a precise statement), Theorem~\ref{maintheorem} gives a complete description of the Martin boundary up to homeomorphism.

\medskip
There is a particularly simple geometric construction of a $Z$-boundary when $\Gamma$ is a nonuniform lattice in the real hyperbolic $n$-space $\mathcal{H}^{n}$. By removing from $\mathcal{H}^{n}$ a $\Gamma$-equivariant collection of disjoint horoballs based at parabolic points, we obtain a $CAT(0)$ space (for the shortest-path metric) on which $\Gamma$ acts cocompactly. One can easily check that the visual boundary of this $CAT(0)$ space is a $Z$-boundary.
In particular, when $n=3$, the $Z$-boundary is homeomorphic to the sphere $\Ss^2$ with a countable and dense set of discs removed.
Thus, the $Z$-boundary is homeomorphic to the Sierpinski carpet.

Dahmani \cite{Dahmani} proposed a construction of a $Z$-boundary for any relatively hyperbolic group with virtually abelian parabolic subgroups, using the so-called coned-off graph of $\Gamma$. See the appendix for details.
Our result implies that both these constructions are equivariantly homeomorphic to the Martin boundary.

\medskip
Our proofs of Theorems~\ref{maintheorem} and~\ref{minimal} use both the deviation inequalities of \cite{GGPY} and the generalization of Ney and Spitzer from \cite{Dussaule}. Roughly stated, we show that the preimage of a parabolic point on the Bowditch boundary, with stabilizer $P$, is homemorphic to the Martin boundary of a neighborhood of $P$ with a finite (though not probability) measure induced by first return times, which by the result of \cite{Dussaule} is a sphere of the appropriate dimension.
Actually the proof is a little bit more technically involved and we now give some more details.

\medskip
To show that the Martin boundary $\partial_{\mu}\Gamma$ is a $Z$-boundary, we have to deal with two types of trajectories, namely those converging to a conical point in the Bowditch boundary and those whose projections converge to a point in the geometric boundary of a parabolic subgroup.

For the first type of trajectories, we use relative Ancona inequalities (Theorem~\ref{Ancona1}).
It was already proved in \cite{GGPY} that such sequences $(g_n)$ converge in the Martin boundary.
Basically, when $(g_n)$ converges to a conical limit point, one can choose an increasing number of transition points on a geodesic from a base point $o$ to $(g_n)$.
This number of transition points tends to infinity.
Relative Ancona inequalities roughly state that the random walk follows geodesic along transition points with large probability.
Forcing the paths to go through an increasing number of bounded neighborhoods of transition points then leads to convergence of the Martin kernels.

For the second type of trajectories, we want to study the induced random walk on a parabolic subgroup $P$.
However, this random walk is not finitely supported so we cannot apply Dussaule's result as it is stated above.
Using properties of the Floyd metric and results of Gerasimov and Potyagailo relating these with the geometry of the Cayley graph, we prove that the induced random walk on a sufficiently large neighborhood of $P$ has large exponential moments.
Using more precise statements of \cite{Dussaule}, we prove that this is enough to conclude: if $(g_n)$ converges in the geometric boundary of $P$, then it converges in the Martin boundary and two different points in the geometric boundary give rise to two different points in the Martin boundary.

\subsection{Organization of the paper}\label{organization}
The paper is divided into five main parts, besides the introduction.
Section~\ref{SectionMartinboundariesofrandomwalks} is devoted to give the necessary probabilistic background on random walks, Markov chains, and their Martin boundaries.

Section~\ref{Sectionrelativelyhyperbolicgroups} is about relatively hyperbolic groups.
In Sections~\ref{SectionrelhypandFloyd} and~\ref{Sectionotherviewpointsonrelhyp} we give different equivalent definitions of those groups and state results of Gerasimov and Potyagailo about the interplay between the Floyd distance and the geometry of the Cayley graph of such groups.
We also define properly what is a $Z$-boundary and what is the geometric compactification of a parabolic subgroup in Section~\ref{Sectioncompactification}.

In Section~\ref{SectionBackgroundMartinboundaries}, we give the necessary geometric background on Martin boundaries for the proof of our main theorem.
In Section~\ref{SectionAnconainequalities}, we state the relative Ancona inequalities obtained in \cite{GGPY}. These inequalities will be used throughout the proofs, especially when we deal with trajectories converging to conical limit points.
In Section~\ref{Sectionthickenedperipheralgroups}, we state the results of Dussaule about Martin boundaries of chains on $\Z^k\times \{1,...,N\}$.
This part is a bit technical and we extend his results to deal later with trajectories converging in the geometric boundary of a parabolic subgroup. Corollary~\ref{markovsummary} could be taken as a black box to prove Proposition~\ref{propconvergencehorospheres} in the following section.

In Section~\ref{SectionconvergenceMartinkernels}, we prove our main theorem, Theorem~\ref{maintheorem}.
We first deal with conical limit points in Section~\ref{Sectionconvergenceconical}, using results of Section~\ref{SectionAnconainequalities} and then with parabolic subgroups in Section~\ref{Sectionconvergenceparabolic}, using results of Section~\ref{Sectionthickenedperipheralgroups}.

In Section~\ref{Sectionminimality}, we prove Theorem~\ref{minimal}, that is, the Martin boundary is minimal.
Again, we will deal separately with trajectories converging to conical limit points and trajectories converging in the boundary of parabolic subgroups.

In the appendix, we give a geometric construction of a $Z$-boundary, using a construction of Dahmani in \cite{Dahmani}.

\subsection{Acknowledgements}
We would like to thank Sebastien Gouëzel for helpful advice. We also thank Wenyuan Yang for helpful conversations.

\section{Martin boundaries of random walks}\label{SectionMartinboundariesofrandomwalks}
Let us give here a proper definition of the Martin boundary and the minimal Martin boundary.
In this paper, we deal with random walks on groups,
but during the proofs, we will restrict the random walk to thickenings of peripheral subgroups and we will not get actual random walks. Thus, we need to define Martin boundaries for more general transition kernels.

Consider a countable space $E$
and equip $E$ with the discrete topology.
Fix some base point $o$ in $E$.
Consider a transition kernel $p$ on $E$ with finite total mass, that is
$p:E\times E\rightarrow \R_+$ satisfies
$$\forall x\in E, \sum_{y\in E}p(x,y)<+\infty.$$
It is often required that the total mass is 1 and in that case, the transition kernel defines a Markov chain on $E$.
In general, we will say that $p$ defines a chain on $E$ and we will sometimes assume that this chain is sub-Markov, that is the total mass is smaller than 1.
If $\mu$ is a probability measure on a finitely generated group $\Gamma$, then $p(g,h)=\mu(g^{-1}h)$ is a probability transition kernel and the Markov chain is the random walk associated to $\mu$.

Define in this context the Green's function $G$ as
$$G(x,y)=\sum_{n\geq 0}p^{(n)}(x,y)\in [0,+\infty],$$
where $p^{(n)}$ is the $n$th convolution power of $p$, i.e.\
$$p^{(n)}(x,y)=\sum_{x_1,...,x_{n-1}}p(x,x_1)p(x_1,x_2)\cdots p(x_{n-1},y).$$

\begin{definition}
Say that the chain defined by $p$ is finitely supported if for every $x\in E$, the set of $y\in E$ such that $p(x,y)>0$ is finite.
\end{definition}

\begin{definition}
Say that the chain defined by $p$ is irreducible if
for every $x,y\in E$, there exists $n$ such that $p^{(n)}(x,y)>0$.
\end{definition}

For a Markov chain, this means that one can go from any $x\in E$ to any $y\in E$ with positive probability.
In this setting, the Green's function $G(x,y)$ is closely related to the probability that a $\mu$-governed path starting at $x$ ever reaches $y$.
Indeed, the strong Markov property shows that the latter quantity is equal to $\frac{G(x,y)}{G(y,y)}$.

Notice that in the case of a random walk on a group $\Gamma$, the Green's function is invariant under left multiplication, so that
$G(x,x)=G(o,o)$ for every $x$.
Thus, up to some constant, $G(x,y)$ is the probability to go from $x$ to $y$.
Moreover, the irreducibility of the chain is equivalent to the condition that the support of the measure $\mu$ generates $\Gamma$ as a semigroup.

\medskip
\emph{In particular, in the context of Theorem~\ref{maintheorem}, the transition kernel defined by the probability measure $\mu$ is irreducible.}

\medskip
We will also use the following definition during our proofs.

\begin{definition}
Say that the chain defined by $p$ is strongly irreducible if
for every $x,y\in E$, there exists $n_0$ such that $\forall n\geq n_0$, $p^{(n)}(x,y)>0$.
\end{definition}

We will also assume that the chain is transient, meaning that
the Green's function is everywhere finite.
For a Markov chain, this just means that almost surely, a path starting at $x$ returns to $x$ only a finite number of times.

Consider an irreducible transient chain $p$. For $y\in E$, define the Martin kernel based at $y$ as
$$K_y(x)=\frac{G(x,y)}{G(o,y)}.$$
The Martin compactification of $E$ with respect to $p$ (and $o$) is a compact space containing $E$ as an open and dense space,
whose topology is described as follows. A sequence $(y_n)$ in $E$ converges to a point in the Martin compactification if and only if
the sequence $K_{y_n}$ converges pointwise.
Up to isomorphism, it does not depend on the base point $o$ and we denote it by $\overline{E}_{\mathcal{M}}$.
We also define the Martin boundary as $\partial_{\mathcal{M}}E=\overline{E}_{\mathcal{M}}\setminus E$.
We refer to \cite{Sawyer} for a complete construction of the Martin compactification.

Seeing the Martin kernel $K$ as a function of two variables $x$ and $y$,
the Martin compactification is then the smallest compact space $M$ in which $E$ is open and dense
and such that $K$ can be continuously extended to the space $E\times M$.
If $\tilde{y}\in \overline{E}_{\mathcal{M}}$, denote by $K_{\tilde{y}}$ the extension of the Martin kernel.

In the particular case of a symmetric Markov chain, that is a Markov chain satisfying $p(x,y)=p(y,x)$,
the Green's distance, which was defined by Brofferio and Blachère in \cite{BrofferioBlachere} as
$$d_G(x,g)=-\mathrm{ln}\mathds{P}(x\rightarrow y),$$
is actually a metric
and the Martin compactification of $E$ with respect to the Markov chain $p$ is
the horofunction compactification of $E$ for this metric.

Now, assume that $E=\Gamma$ is a finitely generated group and that the transition kernel $p$ is defined by a probability measure $\mu$.
In that case, denote by $\overline{\Gamma}_{\mu}$ the Martin compactification and by $\partial_{\mu}\Gamma$ the Martin boundary.
The action by (left) multiplication of $\Gamma$ on itself extends to an action of $\Gamma$ on $\overline{\Gamma}_{\mu}$.

\medskip
One important aspect of the Martin boundary is its relation with harmonic functions.
Recall that if $p$ is a transition kernel on a countable space $E$, a harmonic function is a function $\phi:E \rightarrow \R$ such that $p\phi=\phi$, that is,
$$\forall x \in E, \phi(x)=\sum_{y\in E}p(x,y)\phi(y).$$
We have the following key property (see \cite[Theorem~4.1]{Sawyer}).
\begin{proposition}
Let $p$ be a irreducible transient transition kernel on a countable space $E$.
For any non-negative harmonic function $\phi$, there exists a measure $\nu$ on the Martin boundary $\partial_{\mathcal{M}} E$ of $E$ such that
$$\forall x\in E, \phi(x)=\int_{\partial_{\mathcal{M}} E}K_{\tilde{x}}(x)\mathrm{d}\nu(\tilde{x}).$$
\end{proposition}

Let $\phi$ be a non-negative harmonic function. It is called minimal if any other non-negative harmonic function $\psi$ such that $\psi(x)\leq \phi(x)$ for every $x\in E$ is proportional to $\phi$.
The minimal Martin boundary is the set
$$\partial_{\mathcal{M}}^mE=\{\tilde{x}\in \partial_{\mathcal{M}} E,K_{\tilde{x}}(\cdot) \text{ is minimal harmonic}\}.$$
It is thus a subset of the full Martin boundary $\partial_{\mathcal{M}} E$.
A classical representation theorem of Choquet shows that for any non-negative harmonic function $\phi$, one can choose the support of the measure $\nu$ lying in $\partial_{\mathcal{M}}^mE$.
The measure $\nu$ is then unique
(see the first section of \cite{Sawyer}).
In other words, for any such function $\phi$, there exists a unique measure $\mu_{\phi}$ on $\partial_{\mathcal{M}}^mE$ such that
$$\forall x\in E, \phi(x)=\int_{\partial_{\mathcal{M}}^mE}K_{\tilde{x}}(x)\mathrm{d}\mu_{\phi}(\tilde{x}).$$

\section{Relatively hyperbolic groups}\label{Sectionrelativelyhyperbolicgroups}

\subsection{Relative hyperbolicity and the Floyd metric}\label{SectionrelhypandFloyd}
Let $\Gamma$ be a finitely generated group. The action of $\Gamma$ on a compact Hausdorff space $T$ is called a convergence action if the induced action on triples of distinct points of $T$ is properly discontinuous. 
Suppose $\Gamma\curvearrowright T$ is a convergence action. The set of accumulation points  $\Lambda \Gamma$ of any orbit $\Gamma \cdot x\ (x\in T)$ is called the {\it limit set} of the action. As long as $\Lambda \Gamma$ has more than two points, it is uncountable and the unique minimal closed $\Gamma$-invariant subset of $T$.
The action is then said to be nonelementary. In this case, the orbit of every point in $\Lambda \Gamma$ is infinite. 
The action is {\it minimal} if $\Lambda \Gamma=T$.

{A point $\zeta\in\Lambda \Gamma$  is called {\it conical} if there is a sequence $(g_{n})$ of $\Gamma$ and distinct points $\alpha,\beta \in \Lambda \Gamma$ such that
$g_{n}\zeta \to \alpha$ and $g_{n}\eta \to \beta$ for all $\eta \in  T \setminus\{\zeta\}.$}
The point $\zeta\in\Lambda \Gamma$ is called bounded parabolic if the stabilizer of $\zeta$ in $\Gamma$ is infinite and acts cocompactly on $\Lambda \Gamma \setminus \{\zeta\}$.
The stabilizers of bounded parabolic points are called (maximal) parabolic subgroups.
The convergence action $\Gamma \curvearrowright  T$ is called geometrically finite if every point of $\Lambda \Gamma \subset T$ is either conical or bounded parabolic. Yaman \cite{Yaman} proved that if $\Gamma \curvearrowright T$ is a minimal geometrically finite action, there is a proper Gromov hyperbolic space $X$ on which $\Gamma$ acts by isometries and a $\Gamma$-equivariant homeomorphism $T \to \partial X$.

Suppose now $\Omega$ is a collection of subgroups of $\Gamma$. We say $\Gamma$ is hyperbolic relative to $\Omega$ if there exists some compactum $T$ on which $\Gamma$ acts minimally and geometrically finitely and the maximal parabolic subgroups are the elements of $\Omega$.
Such a compactum is then unique up to $\Gamma$-equivariant homeomorphism \cite{Bowditch} and is called the Bowditch boundary of $(\Gamma, \Omega)$. 
The group $\Gamma$ is said to be relatively hyperbolic if it is hyperbolic relative to some collection of subgroups, or equivalently if it admits a geometrically finite convergence action on some compactum. The group $\Gamma$ is nonelementary relatively hyperbolic if it admits a nonelementary geometrically finite convergence action on some infinite compactum.

A useful fact is the following. Let $\Gamma$ be a group hyperbolic relative to a collection of parabolic subgroups $\Omega$.
The set $\Omega$ is invariant under conjugacy, since the set of parabolic limit points is invariant under translation.
There is a finite number of conjugacy classes of elements of $\Gamma$ (see \cite[Theorem~1B]{Tukia}).

Let $f:\mathbb{R}^{+}\to \mathbb{R}^{+}$ be a function satisfying two conditions: $\sum_{n\geqslant0}f_n<\infty;$ and there exists a  $\lambda\in (0,1)$ such that $1\geqslant f_{n+1}/f_n\geqslant\lambda$ for all $n{\in}\mathbb{N}$. The function $f$ is called the {\it rescaling function.}

Pick a basepoint $o\in \Gamma$ and rescale $C_{S}\Gamma$ by declaring the length of an edge $\sigma$ to be $f(d(o,\sigma))$. The induced shortpath metric on $C_{S}\Gamma$  is called the {\it Floyd metric} with respect to the basepoint $o$ and Floyd function $f$ and denoted by $\delta^{f}_{o}(.,.)$.
Its Cauchy completion (whose topology does not depend on the basepoint) is called the Floyd compactification $\overline{\Gamma}_{f}$ and $\partial_{f}\Gamma= \overline{\Gamma}_{f} \setminus \Gamma$ is called the Floyd boundary.
Karlsson showed that the action of a group on its Floyd boundary is always convergence \cite[Theorem~2]{Karlsson}.
On the other hand, if $\Gamma$ is relatively hyperbolic and if the Floyd function $f$ is not decreasing exponentially too fast, Gerasimov \cite[Map theorem]{Gerasimov} proved that there is continuous $\Gamma$-equivariant surjection ({\it Floyd map}) from the Floyd boundary to the Bowditch boundary. Furthermore,  Gerasimov and Potyagailo \cite[Theorem~A]{GerasimovPotyagailo} proved that the preimage  of any conical point by this map is a singleton and the preimage of a parabolic fixed point $p$ is the limit set for the action of its stabilizer $\Gamma_p$  on  $\partial_f\Gamma$. In particular if $\Gamma_p$ is an amenable non-virtually cyclic group then its limit set on the Floyd boundary is a point. Consequently, when $\Gamma$ is hyperbolic relative to a collection of infinite amenable subgroups which are not virtually cyclic,  the Floyd boundary is homeomorphic to the Bowditch boundary.

\medskip
If $\alpha$ is a (finite or infinite) geodesic in $C_{S}\Gamma$ for the Cayley metric, a point $p\in \alpha$ is said to be {\it $(\epsilon,R)$-deep} if there is a $\gamma \in \Gamma$, $P\in \Omega$ such that the length $R$-segment of $\alpha$ around $p$
is contained in the $\epsilon$-neighborhood of $\gamma P$. Otherwise, $p\in \alpha$ is called an $(\epsilon,R)$-transition point of $\alpha$.
Gerasimov and Potyagailo prove the following key property.
\begin{proposition}\label{Floydgeo}
\cite[Corollary 5.10]{GePoCrelle}
For each $\epsilon>0$, $R>0$ and $D>0$ there is a number $\delta>0$ such that
if $y$ is within word distance $D$ of an $(\epsilon,R)$-transition point of a word geodesic from $x$ to $z$ then $\delta^{f}_{y}(x,z)>\delta$.
\end{proposition}

\subsection{Other viewpoints on relative hyperbolicity}\label{Sectionotherviewpointsonrelhyp}
We defined relative hyperbolicity in terms of convergence actions.
There are two useful equivalent definitions.
The first one is stated in terms of actions on hyperbolic spaces.
\begin{definition}\cite{Bowditch}
A group $\Gamma$ is hyperbolic relative to a collection of subgroups $\Omega$ if it acts by isometries and properly discontinuously on a path-metric proper hyperbolic space $X$ such that the action on $\partial X$ is geometrically finite and the maximal parabolic subgroups are exactly the elements of $\Omega$.
\end{definition}

As stated above, Yaman proved that this notion of relative hyperbolicity coincides with the one we gave in terms of convergence actions.

Relative hyperbolic groups were first introduced by Gromov in \cite{Gromov} as generalizations of finite covolume Kleinian groups with parabolic elements.
The space $X$ plays the role of the hyperbolic space $\mathcal{H}^n$, although it is not $CAT(0)$ in general.
Actually, a large part of the intuition for relative hyperbolic groups comes from this action on a hyperbolic space $X$ and the analogy with Kleinian groups acting on $\mathcal{H}^n$.
During the proofs, it will be more convenient to use the Floyd metric, but one could recover our results using the hyperbolicity of $X$.

The second definition of relative hyperbolicity is based on the geometry of the Cayley graph.
It was introduced by Farb in \cite{Farb}.
We will not use this point of view during the proofs and so we do not include this definition here.
However, we will use Farb's point of view in the construction of a $Z$-boundary in the appendix. We will give more details there.


\subsection{Geometric compactifications}\label{Sectioncompactification}
We now give a precise definition of a $Z$-boundary.
We first define the geometric boundary of an infinite, virtually abelian, finitely generated group.
Let $P$ be such a subgroup, so that there exists a subgroup of $P$ isomorphic to $\Z^k$, for some $k\geq 1$, with finite index in $P$.
Then, any section $P/\Z^k\rightarrow P$ provides an identification between $P$ and $\Z^k\times \{1,...,N\}$ for some $N\geq 1$.
Let $(g_n)$ be a sequence in $P$ and identify $g_n$ with $(z_n,j_n)\in \Z^k\times \{1,...,N\}$.
Say that the sequence $(g_n)$ converges to a point in the boundary of $P$ if $z_n$ tends to infinity and $\frac{z_n}{\|z_n\|}$ converges to some point in the sphere $\Ss^{k-1}$.
This defines what we call the geometric boundary $\partial P$ of $P$.
This construction coincides with the $CAT(0)$ boundary of $P$: the visual boundary of a $CAT(0)$ space $\mathbb{R}^{k}$ on which $P$ acts properly and cocompactly.

More generally, if $F$ is a finite set, we define the geometric boundary of the product $P\times F$ as follows.
First identify $P$ with $\Z^k\times  \{1,...,N\}$ as before. This provides an identification between $P\times F$ and $\Z^k\times  \{1,...,N'\}$ for some other integer $N'\geq 1$.
As above, a sequence $(g_n)$ in $P\times F$ is said to converge in the geometric boundary if its projection onto $\Z^k$ under this identification converges in the $CAT(0)$ boundary of $\Z^k$.
This slight generalization will be useful in the following.
Indeed, for technical reasons, when studying sequences converging in $\partial P$, we will not restrict the random walk to parabolic subgroups but to bounded neighborhoods of them.

\medskip
Suppose now that $\Gamma$ is hyperbolic relative to a collection $\Omega$ of infinite subgroups, each of which is virtually abelian.
For $A\subset \Gamma$ and $g\in \Gamma$, let
$$\mathrm{proj}_{A}(g)=\{h\in A: d(h,g)=d(A,g)\}$$ be the set of closest point projections of $g$ to $A$. 

For another subset $F\subset \Gamma$ let $$\mathrm{proj}_{A}(F)=\cup_{g\in F} \mathrm{proj}_{A}(g).$$
Let $\pi_{A}:\Gamma \to A$ be a function with $\pi_{A}(g)\in \mathrm{proj}_{A}(g)$.

If $P$ is a coset of a parabolic subgroup of $\Gamma$, the diameter of $\mathrm{proj}_{A}(g)$ is finite and bounded independently of $g\in \Gamma$ \cite[Lemma 5.1]{Bowditch}.
In particular, if $(g_n)$ is a sequence in $\Gamma$, the convergence of $\pi_{P}(g_n)$  to the geometric boundary of $P$ does not depend on the choice of $\pi_{P}$.

We will use boundaries throughout the paper. We fix the following terminology.
A \textit{compactification} $\overline{\Gamma}$ of $\Gamma$ is a metrizable compact space, containing $\Gamma$ as an open and dense space, endowed with a group action $\Gamma\curvearrowright \overline{\Gamma}$ that extends the action by left multiplication on $\Gamma$. Then, $\partial \Gamma:=\overline{\Gamma}\setminus \Gamma$ is called a boundary of $\Gamma$.

\begin{definition}\label{Zbdry}
Let $\Gamma$ be a group hyperbolic relative to a collection $\Omega$ of virtually abelian subgroups.
Fix a finite full set of representatives of parabolic subgroups $\Omega_0\subset \Omega$ for the action $\Gamma \curvearrowright \Omega$ by conjugacy.

A $Z$-boundary of $(\Gamma,\Omega)$ is a boundary $\partial \Gamma$ such that the following holds.
\begin{itemize}
\item The identity on $\Gamma$ extends to a continuous equivariant surjective map
$$\Gamma \cup \partial \Gamma \rightarrow \Gamma \cup \partial_B\Gamma,$$
where $\partial_B\Gamma$ is the Bowditch boundary.
\item A sequence $(g_n)$ in $\Gamma$ converges to a point in $\partial \Gamma$ if and only if $(g_n)$ tends to infinity and either $(g_n)$ converges to a conical point in the Bowditch boundary or there exist $g\in \Gamma$ and a parabolic subgroup $P\in \Omega_0$ such that $g^{-1}\pi_{gP}(g_n)$ converges in the geometric boundary of $P$.
\end{itemize}
In other words, a sequence converges in the $Z$-boundary if it converges to a conical limit point or if its projection converges in the geometric boundary of a coset of a parabolic subgroup.
\end{definition}

The term $Z$-boundary was coined by Bestvina \cite{Bestvina} in a more general context. In the context of relatively hyperbolic groups with virtually abelian parabolic subgroups, Dahmani \cite{Dahmani} gave an equivalent explicit construction. In the appendix we show that Dahmani's construction is equivalent to Definition \ref{Zbdry}, justifying our use of the term. 
\begin{lemma}\label{equivalentZbdry}
Let $\Gamma$ be a group hyperbolic relative to a collection $\Omega$ of infinite virtually abelian subgroups.
Up to equivariant homeomorphism, the $Z$-boundary does not depend on the set of  coset representatives.
\end{lemma}

\begin{proof}
Let $\Omega_1$ and $\Omega_2$ be two sets of representatives of conjugacy classes of elements of $\Omega$.
Let $\partial_1\Gamma$ and $\partial_2\Gamma$ be two $Z$-boundaries constructed with $\Omega_1$ and $\Omega_2$ respectively.
If $\gamma\in \partial_1 \Gamma$, there exists a sequence $(g_n)$ in $\Gamma$ that converges to $\gamma$.
If $\gamma$ is mapped to a conical limit point $\alpha$ in the Bowditch boundary, then $(g_n)$ converges in $\alpha$ in the Bowditch compactification.
Thus, it converges to a uniquely defined point $\tilde{\gamma}$ in $\partial_2\Gamma$.

Assume now that $\gamma$ is mapped to a parabolic limit point $\alpha$ in the Bowditch boundary.
Then, $(g_n)$ cannot converges to a conical limit point, so there exists a coset $gP$ of a parabolic subgroup $P\in \Omega_1$, such that the projection $(\pi_{gP}g_n)$ converges in the geometric boundary of $gP$.
There is a unique $h\in\Gamma$ such that $Q=hPh^{-1}\in \Omega_2$.
The points $\pi_{gP}g_n$ lie a bounded distance away from $\pi_{gh^{-1}Q}g_n$, so that $(\pi_{gh^{-1}Q}g_n)$ converges in the geometric boundary of $gh^{-1}Q$.
Thus, there exists a uniquely defined point $\tilde{\gamma}\in \partial_2\Gamma$ such that $(g_n)$ converges to $\tilde{\gamma}$.

This defines a map $\gamma \in \partial_1\Gamma \mapsto \tilde{\gamma} \in \partial_2\Gamma$.
Similarly, one defines a map $\partial_2\Gamma \rightarrow \partial_1\Gamma$ and by construction, the composition of these maps give the identity on $\partial_1\Gamma$ and $\partial_2\Gamma$ respectively.

These maps are continuous.
Indeed, let $(\gamma_n)$ converges to $\gamma$ in $\partial_1\Gamma$. By compactness, we only have to prove that the sequence $(\tilde{\gamma}_n)$ has a unique limit point, which is the image of $\gamma$.
Let $\tilde{\gamma}$ be a limit point, so that there is a subsequence $(\tilde{\gamma}_{\sigma(n)})$ that converges to $\tilde{\gamma}$.
By density, there exist sequences $(g_{n,m})$ that converge to $\gamma_n$ when $m$ tends to infinity, with $g_{n,m}$ in $\Gamma$.
By construction, $(g_{n,m})$ also converges to $\tilde{\gamma_n}$ in $\partial_2\Gamma$.
Both $\Gamma \cup \partial_1\Gamma$ and $\Gamma \cup \partial_2\Gamma$ are metrizable compact spaces.
Denote by $d_1$ and $d_2$ corresponding distances.
Then, one can choose a sequence $h_n:=g_{\sigma(n),\phi(\sigma(n))}$, with $\phi:\N\rightarrow \N$ increasing, such that $d_1(h_n,\gamma_{\sigma(n)})\leq \frac{1}{n}$ and $d_2(h_n,\tilde{\gamma}_{\sigma(n)})\leq \frac{1}{n}$.
Then, $(h_n)$ converges to $\gamma$ in $\partial_1\Gamma$ and converges to $\tilde{\gamma}$ in $\partial_2\Gamma$.
By construction, this proves that $\tilde{\gamma}$ is uniquely determined and is the image of $\gamma$.
Thus, the map $\partial_1\Gamma \rightarrow \partial_2\Gamma$ is continuous and similarly, the map $\partial_2\Gamma\rightarrow \partial_1\Gamma$ is continuous,
which proves these maps are homeomorphisms.
By construction, they also are $\Gamma$-equivariant. 
\end{proof}

According to this lemma, verifying that the Martin boundary satisfies the conditions of Definition \ref{Zbdry} completely determines it up to $\Gamma$-equivariant homeomorphism.

In \cite{Dahmani},
Dahmani gave a geometric construction of a $Z$-boundary for any $\Gamma$ which is hyperbolic relative to virtually abelian subgroups \cite[Theorem 3.1]{Dahmani}. We refer to the appendix for more details.
When $\Gamma$ is a geometrically finite subgroup of isometries of real hyperbolic space $\mathcal{H}^{n}$, Bestvina \cite{Bestvina} noted that $\partial \Gamma$ coincides with the $CAT(0)$ boundary of the space obtained by removing from $\mathcal{H}^{n}$ a $\Gamma$-equivariant family of open horoballs based
at parabolic points of $\partial{\mathcal{H}^{n}}$.






\section{Topology of Martin boundaries}\label{SectionBackgroundMartinboundaries}

\subsection{Ancona's inequality for relatively hyperbolic groups}\label{SectionAnconainequalities}

Suppose $\Gamma$ is a finitely generated group.
Let $\mu$ be a probability measure whose finite support generates $\Gamma$ as a semigroup and let $G$ be the associated Green's function.

Denote by $G(x,z;B^c_R(y))$ the Green's function from $x$ to $z$ conditioned by not visiting the ball of center $y$ and radius $R$, that is
$$G(x,z;B^c_R(y))=\sum_{k\geq 0}\mathbb{P}_x(X_k=z|\forall l\in \{1,...,k-1\}, X_l\notin B_R(y)).$$
When $\Gamma$ is hyperbolic, Ancona \cite{Ancona} proved the following.
\begin{theorem}\label{Ancona-Floyd0}\cite{Ancona}
For every $\epsilon>0$ there is a $R>0$ such that whenever $x,y,z\in \Gamma$ lie along a word geodesic, in that order,
we have $$G(x,z;B^{c}_{R}(y))\leq \epsilon G(x,z).$$ 
\end{theorem}
In other words, if $y$ is on a word geodesic connecting $x$ and $z$, a random path between
$x$ and $z$ passes in a bounded neighborhood of $y$ with high probability.
Ancona used this to identify the Martin boundary of hyperbolic groups with their Gromov boundary.
To do the same for relatively hyperbolic groups, we will need the following result of
Gekhtman, Gerasimov, Potyagailo, and Yang.
Let $f$ be a Floyd function.
\begin{theorem}\label{Ancona-Floyd}\cite[Theorem 5.2]{GGPY}
For each $\epsilon>0$ and $\delta>0$ there is a $R>0$ such that for all $x,y,w\in \Gamma$ with $\delta^{f}_{w}(x,y)>\delta$ one has
$$G(x,y;B^{c}_R(w))\leq \epsilon G(x,y).$$
\end{theorem}

The analogy with Ancona's estimate along geodesics comes from the following fact.
Recall that if $\alpha$ is a word geodesic in $\Gamma$, a point $p\in \alpha$ is called an $(\epsilon,R)$-transition point if there does \emph{not} exist a coset $P$ of a parabolic subgroup such that the length $R$-segment of $\alpha$ surrounding $p$ is contained in the $\epsilon$-neighborhood of $P$.
\begin{theorem}\label{Ancona1}\cite{GGPY}
Let $\alpha$ be a geodesic in $C_{S}\Gamma$.
Then for any $\epsilon>0$, $r>0$, $D>0$, and $\eta>0$ there is an $R>0$ such that if $x,y,z\in \Gamma$, $\alpha$ is a geodesic between $x$ and $z$, and $y$ is within $D$ of an $(\eta,r)$-transition point of $\alpha$ then
$$G(x,z;B^{c}_R(y))\leq \epsilon G(x,z).$$
\end{theorem}

Indeed, in this situation, the Floyd distance $\delta_y^f(x,z)$ is bounded from below by a universal constant, according to Proposition~\ref{Floydgeo}.

Although Theorem \ref{Ancona-Floyd} holds for arbitrary finitely generated groups, its results are vacuous when the Floyd boundary is trivial, and besides virtually cyclic groups, the only known examples with nontrivial Floyd boundary are nonelementary relatively hyperbolic groups.

\subsection{Martin boundaries of thickened abelian groups}\label{Sectionthickenedperipheralgroups}
To understand the behavior of $K_{g_n}(g)$, when $g_n$ converges in the geometric boundary of a parabolic subgroup, we will introduce the transition kernel of the first return to the corresponding subgroup $P$.
We will then get a sub-Markov chain on $P$ and we will show that we can identify this first-return-chain with a $\Z^k$-invariant sub-Markov chain on $\Z^k\times \{1,...,N\}$
(see Lemma~\ref{lemmasubMarkovneighborhoodhorosphere}).
We will then use results for such chains.

In \cite{Dussaule}, the author shows that under some technical assumptions, the Martin boundary of such a chain on $\Z^k\times \{1,...,N\}$ coincides with the geometric boundary.
In this setting, the geometric boundary is defined as follows.
A sequence $(z_n,j_n)$ in $\Z^k\times \{1,...,N\}$ converges to a point in the geometric boundary if $z_n$ tends to infinity and $\frac{z_n}{\|z_n\|}$ converges in the unit sphere $\mathbb{S}^{k-1}$.
We now introduce the assumptions of \cite{Dussaule} and we will later show that they are satisfied in our setting.

Consider a $\Z^k$-invariant chain $p$ on the product space $\Z^k\times \{1,...,N\}$.
For every function defined on $\mathbb{Z}^k\times \{1,...,N\}$, the $\{1,...,N\}$ coordinate will be considered as an index.
For example, the transition kernel will be written as $p_{j_1,j_2}(z_1,z_2)$, its $n$th power of convolution as $p_{j_1,j_2}^{(n)}(z_1,z_2)$, the Green's function as $G_{j_1,j_2}(z_1,z_2)$ and the Martin kernel as $K_{j_1,j_2}(z_1,z_2)$.
We can thus see these functions as the entries of $N\times N$ matrices.
Assume that the chain $p$ is strongly irreducible, that is,
for every $j_1,j_2\in \{1,...,N\}$ and for every $z_1,z_2\in \Z^k$, there exists $n_0$ such that for every $n\geq n_0$, $p^{(n)}_{j_1,j_2}(z_1,z_2)>0$.
As we will see later (see Lemma~\ref{Spitzertrick}), strong irreducibility is not too much to ask and we will be able to reduce our study of irreducible chains to strongly irreducible ones.

In \cite{NeySpitzer}, Ney and Spitzer show that the Martin boundary of a strongly irreducible, finitely supported, noncentered random walk on $\mathbb{Z}^k$ coincides with the $CAT(0)$ boundary.
Their proof is based on the study of minimal harmonic functions which are of the form $z\in \mathbb{Z}^k\mapsto \mathrm{e}^{u\cdot z}$ for some $u\in \mathbb{R}^k$ satisfying the condition
\begin{equation}\label{equationharmonic}
\sum_{z\in \mathbb{Z}^k}p(0,z)\mathrm{e}^{u\cdot z}=1.
\end{equation}
In our setting,
for $u\in \mathbb{R}^k$, we define the $N\times N$ matrix $F(u)$ whose entries are given by
$$F_{j_1,j_2}(u)=\sum_{z\in \mathbb{Z}^k}p_{j_1,j_2}(0,z)\mathrm{e}^{u\cdot z}.$$
The entries of this matrix may be infinite.
We restrict our attention to the set where they are finite and denote this set by $\mathcal{F}_0$:
$$\mathcal{F}_0=\{u\in \mathbb{R}^k, \forall j_1,j_2\in \{1,...,N\},F_{j_1,j_2}(u)<+\infty\}.$$
We also denote by $\mathcal{F}$ the interior of $\mathcal{F}_0$.

\begin{lemma}
For every $u\in \mathcal{F}_0$, the matrix $F(u)$ has non-negative entries. Furthermore, this matrix is strongly irreducible, meaning that there exists $n\geq 0$ such that $F(u)^n$ has positive entries. 
\end{lemma}

\begin{proof}
Direct calculation shows that the entries of $F(u)^n$ are given by
$$F_{j_1,j_2}(u)^{n}=\sum_{z\in \mathbb{Z}^k}p^{(n)}_{j_1,j_2}(0,z)\mathrm{e}^{u\cdot z}.$$
Strong irreducibility of $F(u)$ is deduced from strong irreducibility of $p$.
\end{proof}

Since $F(u)$ is strongly irreducible, it follows from the Perron-Frobenius theorem (see \cite[Theorem 1.1]{Seneta}) that $F(u)$ has a dominant positive eigenvalue, that is an eigenvalue $\lambda(u)$ which is positive and such that for every other eigenvalue $\lambda\in \mathbb{C}$, $|\lambda|<\lambda(u)$.
Moreover, any eigenvector associated to $\lambda(u)$ has non-zero coordinates and we can assume that every coordinate is positive.
The analog of Equation (\ref{equationharmonic}) will be 
\begin{equation}\label{equationharmonicthickening}
\lambda(u)=1.
\end{equation}

Denote by $D$ the set where $\lambda(u)$ is at most 1:
$D=\{u\in \mathcal{F}, \lambda(u)\leq 1\}.$
The two technical assumptions of \cite{Dussaule} on the chain $p$ are the following.

\begin{assumption}\label{Assumption1}
The set $D$ is compact.
\end{assumption}

\begin{assumption}\label{Assumption2}
The minimum of the function $\lambda$ is strictly smaller than 1.
\end{assumption}

Since $\lambda(u)$ is a dominant eigenvalue, it is analytic in $u$ (see Proposition~8.20 in \cite{Woess}).
For $u\in \mathcal{F}$, denote by $\nabla \lambda(u)$ the gradient of $\lambda$ with respect to $u$.
We have the following (see Lemma~3.13 in \cite{Dussaule}).
\begin{lemma}\label{explicithomeo}
Under Assumptions~\ref{Assumption1} and~\ref{Assumption2}, the set $\{u\in \mathbb{R}^k,\lambda(u)=1\}$ is homeomorphic to $\mathbb{S}^{k-1}$ and
an explicit homeomorphism is given by
$$u\in \{u\in \mathbb{R}^k,\lambda(u)=1\}\mapsto \frac{\nabla \lambda(u)}{\|\nabla \lambda (u)\|}.$$
\end{lemma}

This homeomorphism provides a homeomorphism $\varphi$ between the geometric boundary of $\Z^k\times \{1,...,N\}$ and $\mathbb{S}^{k-1}$ constructed as follows.
Let $(z_n,j_n)$ be a sequence in $\Z^k\times \{1,...,N\}$ converging to a point $\tilde{z}$ in the geometric boundary $\partial (\Z^d\times \{1,...,N\})$.
Then $z_n$ tends to infinity and $\frac{z_n}{\|z_n\|}$ converges to a point $\theta$ in the unit sphere $\mathbb{S}^{k-1}$.
There exists a unique $u\in \{u\in \mathbb{R}^k,\lambda(u)=1\}$ such that $\theta =  \frac{\nabla \lambda(u)}{\|\nabla \lambda (u)\|}$.
Then, define $\varphi(\tilde{z})=u$.

The Martin boundary is defined up to the choice of a base point. Fix such a base point $(z_0,j_0)\in \Z^k\times \{1,...,N\}$.
Now, we can state that the Martin boundary coincides with the geometric boundary (see Proposition~3.17 in \cite{Dussaule}).

\begin{proposition}\label{latticeMartinboundary}
Let $p$ be a strongly irreducible transition kernel on $\Z^k\times \{1,...,N\}$ which is $\Z^k$-invariant and satisfies Assumptions~\ref{Assumption1} and~\ref{Assumption2}.
If $z_n\in \Z^k$ converges to $\tilde{z}\in \partial\Z^k$, let $u=\varphi(\tilde{z})$.
Then, for every $z\in \Z^k$ and for every $j_1,j_2\in \{1,...,N\}$, there exists $C_{j_1}>0$ which only depends on $j_1$ such that
$K_{j_1,j_2}(z,z_n)$ converges to $C_{j_1}\mathrm{e}^{u\cdot (z-z_0)}$.
\end{proposition}

Consider now a chain $p$ on $\Z^k\times \N$.
If $N\geq1$, define the induced chain $p_N$ as the chain of the first return to $\Z^k\times \{1,...,N\}$,
that is, if $(z,j),(z',j')\in \Z^k\times \{1,...,N\}$,
\begin{align*}
&p_N((z,j),(z',j'))=p((z,j),(z',j'))\\
&+\sum_{k\geq 1}\underset{\underset{j1,...,j_{k}>N}{(z_1,j_1),...,(z_{k},j_{k})}}{\sum}p((z,j),(z_1,j_1))p((z_1,j_1),(z_2,j_2))...p((z_k,j_k)(z,j')).
\end{align*}

Denote by $G$ the Green's function associated to $p$ and by $G_N$ the Green's function associated to the induced chain $p_N$.
Then, we have the following lemma.

\begin{lemma}\label{sameGreenDussaule}
The restriction to $\Z^k\times\{1,...,N\}$ of the Green's function $G$ coincides with the Green's function $G_N$.
\end{lemma}

\begin{proof}
Let $(z,j),(z',j')\in \Z^k\times \{1,...,N\}$.
By definition,
\begin{align*}
&G_N((z,j),(z',j'))=\sum_{n\geq 0}p_N^{(n)}((z,j),(z',j'))\\
&=\sum_{n\geq 0}\underset{j_i\leq N}{\underset{(z_1,j_1),...,(z_{n},j_{n})}{\sum}}
p_N((z,j),(z_1,j_1))p_N((z_1,j_1),(z_2,j_2))\cdots p_N((z_{n},j_{n}),(z',j'))\\
&+p_N^{(0)}((z,j),(z',j')).
\end{align*}

We use the notation $\gamma_i=(z_i,j_i)$.
Since $p_N$ is the transition kernel of the first return to $\Z^k\times \{1,...,N\}$,
$$p_N(\gamma_i,\gamma_{i+1})=\sum_{m\geq 0}\underset{j_i^{(l)}>N}{\underset{\gamma_i^{(1)},...,\gamma_i^{(m-1)}}{\sum}}p(\gamma_i,\gamma_i^{(1)})p(\gamma_i^{(1)},\gamma_i^{(2)})\cdots p(\gamma_i^{(m-1)},\gamma_{i+1}),$$
where $j_i^{(l)}$ is the projection of $\gamma_i^{(l)}\in \Z^k\times \N$ on the $\N$ factor.

Reorganizing the first sum, one gets
$$G_N((z,j),(z',j'))=\sum_{n\geq 0}\underset{\gamma_i\in \Z^k\times \N }{\underset{\gamma_1,...,\gamma_{n-1}}{\sum}}p(\gamma,\gamma_1)p(\gamma_1,\gamma_2)\cdots p(\gamma_{n-1},\gamma')=G(\gamma,\gamma').$$

In other words, every trajectory from $(z,j)$ to $(z',j')$ for the initial chain $p$ defines a trajectory from $(z,j)$ to $(z',j')$ for $p_N$, excluding every point of the path that is not in $\Z\times \{1,...,N\}$
and every trajectory for $p_N$ is obtained in such a way.
Summing over all trajectories, the two Green's functions coincide.
\end{proof}

Let $M\geq 0$.
If $p$ is a chain on $\Z^k\times \{1,...,N\}$, say that $p$ has exponential moments up to $M$ if for every $j,j'\in \{1,...,N\}$,
$$\sum_{z\in \mathbb{Z}^k}p_{j,j'}(0,z)\mathrm{e}^{M\|z\|}<+\infty.$$

We also have the following proposition. 


\begin{proposition}\label{inducedchainassumption1}
Let $p$ be a $\Z^k$-invariant, finitely supported, strongly irreducible transition kernel on $\Z^k\times \N$.
Then, there exist $N_0\geq0$ and $M\geq 0$ 
such that whenever $N\geq N_0$ and the chain $p_N$ has exponential moments up to $M$, $p_{N}$ satisfies Assumption~\ref{Assumption1}.
\end{proposition}

\begin{proof}
We will prove that there exist $u_0\in \R^k$ and $M\geq 0$, independent of large enough $N$, such that if the chain $p_N$ has exponential moments up to $M$, then
\begin{equation}\label{equationlambdauleq2}
\{u\in \mathbb{R}^k, \lambda(u)\leq 2\} \subset \{u\in \mathbb{R}^k, \|u\|\leq u_0\} \subset \mathcal{F}.
\end{equation}
Denote by $(e_1,...,e_k)$ the canonical basis in $\mathbb{R}^k$ and fix $j_0\in \{1,...,N\}$.
Since the chain $p$ is strongly irreducible, there exists $n_i$ such that for every $n\geq n_i$, $p^{(n)}((0,1),(e_i,1))>0$, so that there is a path of length $n$ from $(0,1)$ to $(e_i,1)$.
This path stays in $\Z^k\times \{1,...,N_i\}$ for some $N_i$, since the chain $p$ is finitely supported.
Thus, for every $N\geq N_i$, the restricted chain $p_{N_i}$ satisfies $p_{N_i}^{(n)}((0,1),(e_i,1))>0$, if $n\geq n_i$.
Thus, there exist
$a>0$, $N_0$ and $n_0\geq 0$ such that for every $i\in \{1,...,k\}$,
$p_N^{(n)}(0,e_i)\geq a$ for $N\geq N_0$, $n\geq n_0$ and $a$ does not depend on $N_0$.

Now, let $u\in \mathbb{R}^d$ and fix $L\geq0$.
There exists $u_0$ such that if $\|u\| \geq u_0$, then at least one of the $\mathrm{e}^{u\cdot e_i}$ will be larger than $\frac{L}{a}$, so that $F_{j_0,j_0}(u)^{n_0}\geq L$.
Since $a$ does not depend on $N_0$, $u_0$ does not depend on $N_0$.
Moreover, if $p$ has exponential moments up to $u_0+1$, then $F(u)$ has finite entries for $u_0\leq \|u\|\leq u_0+1$ and so does $F(u)^{n_0}$.
Let $v(u)$ be an eingenvector associated to $\lambda(u)$.
Then, it is an eigenvector of $F(u)^{n_0}$ associated to $\lambda(u)^{n_0}$.
Since $F(u)$ is strongly irreducible, $v(u)$ has non-zero coordinates and we can even choose $v(u)$ with strictly positive coordinates.
Denote by $v(u)(j_0)$ its $j_0$th coordinate.
Then, $v(u)(j_0)\lambda(u)^{n_0}\geq F_{j_0,j_0}(u)^{n_0}v(u)(j_0)$ so that $\lambda(u)^{n_0}\geq F_{j_0,j_0}(u)^{n_0}\geq L$.

Consequently, $\lambda(u)^{n_0}$, hence $\lambda(u)$, can be made arbitrarily large, when enlarging $\|u\|$.
Moreover, if $p$ has sufficiently large exponential moments, then $\lambda(u)$ is well defined for arbitrarily large $\|u\|$.
Equation (\ref{equationlambdauleq2}) follows from these two facts.
This proves that the sub-level $\lambda(u)\leq 1$, is bounded, thus compact and contained in the open set $\lambda(u)<2$, which is included in $\mathcal{F}$.
Thus, Assumption~\ref{Assumption1} is satisfied.
\end{proof}

\begin{remark}
It will be important in further applications that the $M$ in the statement of this proposition is independent of $N$.
Indeed, in the next section we will be in a similar situation and we will prove that for every $K$, for large enough $N$, the chain $p_N$ has exponential moments up to $K$.
We will then choose $N_1$ so that if $N\geq N_1$, the chain $p_N$ has exponential up to $M$, so that for $N\geq N_0,N_1$, the chain $p_N$ satisfies Assumption~\ref{Assumption1}.
\end{remark}
We will also use the following.

\begin{lemma}\label{inducedchainassumption2}
Let $p$ be a $\Z^k$-invariant, strongly irreducible chain on $\Z^k\times \{1,...,N\}$.
If $p$ is (strictly) sub-Markov, then it satisfies Assumption~\ref{Assumption2}.
\end{lemma}

\begin{proof}
The fact that the chain is strictly sub-Markov
means that the matrix $F(0)$ defined above is strictly sub-stochastic.
In particular, its dominant eigenvalue $\lambda(0)$ satisfies $\lambda(0)<1$ and the minimum of $\lambda$ is strictly less than 1, so Assumption~\ref{Assumption2} is satisfied.
\end{proof}

\medskip
Combining the explicit formula given in Proposition~\ref{latticeMartinboundary} together with Proposition~\ref{inducedchainassumption1} and Lemma~\ref{inducedchainassumption2}, we obtain the following corollary which describes convergence in the Martin boundary for a chain on $\Z^k\times \N$.
It will be important in the proof of Theorem~\ref{maintheorem} in the next section.
Indeed, we will identify the random walk on $\Gamma$ with a $\Z^k$-invariant Markov chain on $\Z^k \times \N$. 
To identify the Martin boundary of the induced chain $p_N$, we will verify the conditions of this corollary. 

\begin{corollary}\label{markovsummary}
Let $p$ be a $\Z^k$-invariant, finitely supported, strongly irreducible transition kernel on $\Z^k\times \N$ such that: 

a)For large enough $N$, the induced chain $p_N$ on $\Z^k\times \{1,...,N\}$ is strictly sub-Markov.

b)For all $M$ there exists an $N_0>0$ such that for $N>N_0$, the chain $p_N$ has exponential moments up to $M$.

Then, a sequence $(z_{n},j_{n})$ in $\Z^{k}\times \N$, with $\mathrm{sup}(j_n)<+\infty$, converges to a point in the Martin boundary of $p$ if and only if $\|z_n\|$ tends to infinity and $\frac{z_n}{\|z_n\|}$ converges to a point of $\mathbb{S}^{k-1}$.
\end{corollary}

Thus, the fact that $(z_{n},j_{n})$ converges to a point in the Martin boundary does not depend on the sequence $(j_n)$.
In particular, $(z_{n},j_{n})$ converges in the Martin boundary if and only if its projection $(z_n,0)$ on $\Z^k\times \{0\}$ converges in the Martin boundary and the limits are the same.

\section{Convergence of Martin kernels: proof of Theorem \ref{maintheorem}}\label{SectionconvergenceMartinkernels}
Let $\Gamma$ be hyperbolic 
relative to a collection $\Omega$ of infinite virtually abelian subgroups.
Let $\mu$ be a measure on $\Gamma$ whose finite support generates $\Gamma$ as a semigroup.
In this section we prove that the Martin boundary is a $Z$-boundary, proving Theorem~\ref{maintheorem}.
Recall that $\partial P$ denotes the geometric boundary of a maximal parabolic subgroup $P$ defined in Section~\ref{Sectioncompactification}.
We fix a finite set $\Omega_0$ of representatives of conjugacy classes of $\Omega$.
We will deal separately with sequences converging to conical limit points and sequences converging in $g\partial P$ for some coset $gP$ of a parabolic subgroup $P$.
For the second case, we will apply results of Section~\ref{Sectionthickenedperipheralgroups}.
It will be more convenient to deal with strongly irreducible chains.
Thus, we first show that we can reduce to such chains.

In a very general context, consider a chain $p$ on a countable space $E$.
Define the modified chain $\tilde{p}$ on $E$ by
$$\tilde{p}(x_1,x_2)=\frac{1}{2}\Delta(x_1,x_2)+\frac{1}{2} p(x_1,x_2),$$
where $\Delta(x_1,x_2)=0$ if $x_1\neq x_2$ and 1 otherwise.
Denote by $\tilde{p}^{(n)}$ the $n$th convolution power of $\tilde{p}$. Also denote by $\tilde{G}$ the associated Green's function:
$$\tilde{G}(x_1,x_2)=\sum_{n\geq 0}\tilde{p}^{(n)}(x_1,x_2).$$
We have the following (see \cite[Lemma~3.20]{Dussaule}).

\begin{lemma}\label{Spitzertrick}
With these notations, $\frac{1}{2}\tilde{G}(x_1,x_2)=G(x_1,x_2)$ and thus the Martin kernels are the same.
\end{lemma}

In our context, this means we can assume that $\mu(e)>0$.
Now, since the random walk is irreducible, if it satisfies $\mu(e)>0$, it is strongly irreducible.

\subsection{Convergence to conical limit points}\label{Sectionconvergenceconical}
We first study conical limit points.
We prove the following.

\begin{proposition}\label{Martinboundaryconicalpoints}
Consider a sequence $(g_n)$ of $\Gamma$ that converges to a conical point $\alpha$ in the Bowditch boundary.
Then, $(g_n)$ converges to a point in the Martin boundary.
\end{proposition}

This is a consequence of the following results of \cite{GGPY}, which were recalled in the introduction.

\begin{proposition}\label{MartintoBowditch1}\cite[Theorem~6.3]{GGPY}
The identity map $\Gamma \rightarrow \Gamma$ extends to a continuous equivariant surjection $F$ from the Martin compactification to the Bowditch compactification.
\end{proposition}

\begin{proposition}\label{MartintoBowditch2}\cite[Corollary~6.13]{GGPY}
The preimage $F^{-1}(\alpha)$ of a conical limit point $\alpha$ consists of a single point.
\end{proposition}


In the case of a free product, the analog of conical limit points are infinite words.
In \cite{Dussaule}, the author also proves Proposition~\ref{Martinboundaryconicalpoints} in this context but the proof is simpler and slightly differs from the proofs in \cite{GGPY}.
In both cases, the key point is that the random walk tracks relative geodesics when it converges to conical limit points.
This idea is encoded in relative Ancona inequalities in \cite{GGPY} (see Theorem~\ref{Ancona1}) and in the use of transitional sets in \cite{Dussaule}.

\subsection{Convergence in parabolic subgroups}\label{Sectionconvergenceparabolic}
In this section, we prove the convergence of the Martin kernels $K_{g_n}(\cdot)$ when $(g_n)$ converges in the boundary of a coset of a parabolic subgroup.
We begin by fixing notations.
Recall that we fixed a finite set $\Omega_0$ of conjugacy classes of parabolic subgroups such that every maximal parabolic subgroup is conjugated to one in $\Omega_0$.
Fix a parabolic subgroup $P\in \Omega_0$ and
assume that $(g_n)$ is such that the closest projection $\pi_{P}g_n$ of $g_n$ in $P$ converges to a point in $\partial P$.
We first recall the definition of the geometric boundary of $P$ from Section~\ref{Sectioncompactification}.

Parabolic subgroups are assumed to be finitely generated virtually abelian groups.
Denote by $k$ the rank of $P$, that is, assume that $P$ contains a subgroup isomorphic to $\Z^k$ with finite index.
Any section $P/\Z^k\rightarrow P$ provides an identification between $P$ and $\Z^k\times \{1,...,N\}$.

Identify then $g_n$ with $(z_n,j_n)\in \Z^k\times \{1,...,N\}$.
By definition, the sequence $(g_n)$ converges to a point in the boundary $\partial P$ of $P$ if $z_n$ tends to infinity and $\frac{z_n}{\|z_n\|}$ converges to some point in $\Ss^{k-1}$.
Denote the corresponding point in $\Ss^{k-1}$ by $\theta$ and say that $(g_n)$ converges to $\theta$.

\begin{proposition}\label{propconvergencehorospheres}
If $(g_{n})$ is a sequence such that $(\pi_{P}g_{n})$ converges to a point of $\partial P$, then $(g_{n})$ and $(\pi_{P}g_{n})$ both converge in the Martin boundary to the same point.
\end{proposition}

Since a sequence $(g_n)$ converges in $g\partial P$ if and only if $(g^{-1}g_n)$ converges in $\partial P$ and since
$(g_n)$ converges in the Martin boundary if and only if $(g^{-1}g_n)$ converges in the Martin boundary, the same result will hold for any coset $gP$ of $P$.

To prove this proposition, we will adapt the arguments given in \cite{Dussaule} to our situation.
We will follow the same strategy as in Section~5 of \cite{Dussaule}.
Namely, we will first prove the proposition assuming $(g_n)$ stays in a fixed neighborhood of $P$.
Finally, we will use the relative Ancona inequalities to reduce to the case where $(g_n)$ does stay in such a neighborhood.

\begin{proposition}\label{propconvergenceneighborhoodhorospheres}
Suppose $(g_n)$ is a sequence with $\sup_{n}d(g_{n},P)<\infty$. Then $(g_{n})$ converges to a point in the Martin boundary $\partial_{\mu}\Gamma$ if and only if $(\pi_{P}g_{n})$ converges to a point of $\partial P$ and in that case, the limit of $(g_n)$ and $(\pi_Pg_n)$ in the Martin boundary are the same.
\end{proposition}

Let $\eta\geq 0$.
We denote the $\eta$-neighborhood of $P$ by $N_{\eta}(P)$.
To prove this proposition, we introduce the chain $p$ corresponding to the first return to $N_{\eta}(P)$, defined as in Section~\ref{SectionMartinboundariesofrandomwalks}, that is,
for $\gamma_1,\gamma_2\in N_{\eta}(P)$ denote by $p(\gamma_1,\gamma_2)$ the probability that the $\mu$-random walk starting at $\gamma_1$ returns to $N_{\eta}(P)$, and first does so at $\gamma_2$.
In other words 
$p(\gamma_{1},\gamma_{2})=G(\gamma_{1},\gamma_{2};N_{\eta}^{c}(P))$.
We will see that the probability that the $\mu$-random walk starting at $\gamma_1$ never goes back to $N_{\eta}(P)$ is positive (see Lemma~\ref{lemmasubMarkovneighborhoodhorosphere}).
Thus, $p$ is not a probability transition kernel
and defines a sub-Markov chain on $N_{\eta}(P)$.
Nevertheless, one can still define the Green's function associated to $p$ as
$$G_p(\gamma_1,\gamma_2)=\sum_{n\geq 0}p^{(n)}(\gamma_1,\gamma_2), \gamma_1,\gamma_2\in  N_{\eta}(P),$$
where $p^{(n)}$ is the $n$th power of convolution of $p$.

\begin{lemma}\label{sameGreen}
The Green's function $G_p$ coincides with the restriction to $N_{\eta}(P)$ of the Green's function $G_{\mu}$ associated to the initial random walk.
\end{lemma}

\begin{proof}
This is given by Lemma~\ref{sameGreenDussaule}.
Recall that the proof is based on the following.
Every trajectory from $\gamma$ to $\gamma'$ for the random walk defines a trajectory from $\gamma$ to $\gamma'$ for $p$, excluding every point of the path that is not in $N_{\eta}(P)$
and every trajectory for $p$ is obtained in such a way.
Summing over all trajectories, the two Green's functions coincide.
\end{proof}

We also have the following property.

\begin{lemma}\label{strongirreducibility}
The chain $p$ is strongly irreducible.
\end{lemma}

\begin{proof}
The proof is based on the same idea as the proof of Lemma~\ref{sameGreen}.
First, the initial random walk is irreducible.
Now, every trajectory for $p$ comes from a trajectory for the random walk on the whole group, after excluding points that do not stay in the neighborhood of $P$.
Thus, there is a positive proportion (for $p$) of paths from any point $\gamma \in N_{\eta}(P)$ to any other point $\gamma'\in N_{\eta}(P)$.
This proves that $p$ is irreducible.
Now, recall that we assumed that $\mu(e)>0$ (see Lemma~\ref{Spitzertrick}), so that $p(\gamma,\gamma)>0$ and thus $p$ is strongly irreducible.
\end{proof}

In light of Lemma \ref{sameGreen}, to prove Proposition \ref{propconvergenceneighborhoodhorospheres}, it suffices to show that a sequence satisfying its conditions converges to a point in the Martin boundary of $N_{\eta}(P)$ with the induced process $p$ (for $\eta$ large enough).

We first notice that, as a set, $\Gamma$ can be identified $P$-equivariantly with $P\times \N$.
Indeed, $P$ acts by left multiplication on $\Gamma$ and the quotient is countable.
We order elements in the quotient according to their distance to $P$.
It follows that the $\eta$-neighborhood $N_{\eta}(P)$ can be $P$-equivariantly identified with $P\times \{1,...,N\}$.
Moreover, if $\eta'\leq \eta$, the set $P\times \{1,...,N'\}$ identified with $N_{\eta'}(P)$ is a subset of $P\times \{1,...,N\}$ identified with $N_{\eta}(P)$.

Now, identifying $P$ with $\Z^k\times F$, where $F$ is finite,
the group $\Gamma$ can be identified with $\Z^k\times \mathbb{N}$.
Thus, the $\mu$-random walk can be considered
as a $\Z^{k}$-invariant Markov chain $q$ on $\Z^{k}\times \N$ and the restriction of the random walk to $N_{\eta}(P)$ coincides with the restriction of the chain $q$ to $\Z^k\times \{1,...,N\}$.

Let $(\gamma_n)$ be a sequence in $N_{\eta}(P)$ and identify $\gamma_n$ with $(z_n,j_n)\in \Z^k\times \{1,...,N\}$.
Notice that the projection of $(\gamma_n)$ to $P$ converges in the geometric boundary $\partial P$ of $P$ if and only if $((z_n,j_n))$ converges in the geometric boundary of $\Z^k\times \{1,...,N\}$, since in both cases, the sequence converges in the geometric boundary if and only if $(z_n)$ tends to infinity and $\frac{z_n}{\|z_n\|}$ converges to a point in the sphere.

To prove Proposition \ref{propconvergenceneighborhoodhorospheres}, it suffices to show that the Markov chain $q$ on $\Z^{k}\times \N$ and its induced chain $p$ on $\Z^{k}\times \{1,..,N\}$  satisfy the conditions of Corollary \ref{markovsummary}. 
Thus, we just need to  show that for large enough $\eta$, the induced chain on $N_{\eta}(P)$ has sufficiently large exponential moments
and is strictly sub-Markov.

\begin{lemma}\label{lemmasubMarkovneighborhoodhorosphere}
The induced chain $p$ is strictly sub-Markov.
\end{lemma}

\begin{proof}
It suffices to show that there exists $\gamma \in N_{\eta}(P)$ such that
$$\sum_{\gamma'\in  N_{\eta}(P)}p(\gamma ,\gamma')<1.$$
This follows from the fact that the $\mu$-random walk starting at $\gamma$ with $d(\gamma,P)=\eta$ has a positive probability of never returning to $N_{\eta}(P)$.
This, in turn, follows from the fact that the random walk almost surely converges to a conical point (see for example \cite[Theorem~9.8, Theorem~9.14]{GGPY}).
\end{proof}



For $M\geq 0$, recall that $p$ is said to have exponential moments up to $M$ if for every $j, j'\in \{1,...,N\}$,
$$\sum_{z\in \mathbb{Z}^k}p_{j,j'}(0,z)\mathrm{e}^{M\|z\|}<+\infty.$$
\begin{proposition}\label{exponentialmoments}
Let $M\geq 0$.
For large enough ${\eta}$, $p$ has exponential moments up to $M$.
\end{proposition}

The proof of Proposition~\ref{exponentialmoments} will be divided into several steps.
We will use the following two results which are related to hyperbolic properties of the geometry of the Cayley graph $C_S\Gamma$. Every maximal parabolic subgroup $P$ of a relatively hyperbolic group  is quasiconvex (see for example \cite[Corollary~3.9]{GerasimovPotyagailo}) so we have the following.
\begin{lemma}\label{lemmacontraction1}\cite[Lemma 5.1]{Bowditch}
 There exists a constant $c_0$ such that the following holds.
If $g_1,g_2\in \Gamma$ and if $g'_i\in \mathrm{proj}_{P}g_{i}$
for $i=1,2$
then, 
$$d(g'_1,g'_2)\leq d(g_1,g_2) + c_0.$$
\end{lemma}

\begin{lemma}\label{lemmacontraction2}
There exists an $a_{0}>0$ such that the function $\rho:\mathbb{R}_+\rightarrow \mathbb{R}_+$ defined by 
$$\rho(\eta)=\inf \{d(g_{1},g_{2}):d(\pi_{P}(g_{1}),\pi_{P}(g_{2}))\geq a_{0}, d(g_{i},P)>\eta\}$$ tends to infinity as $\eta \to \infty$.
\end{lemma}

\begin{proof}
By Proposition 8.5 of \cite{GePoCrelle}, there are constants $a_{0},D>0$ (independent of the parabolic subgroup $P$) such that if $\alpha$ is a geodesic with $d(\alpha,P)\geq D$, then $$diam(\mathrm{proj}_{P}(\alpha))<a_0.$$ 
Now, consider $g_{1},g_{2}\in \Gamma$ with $d(g_{1},P)>\eta$ and $d(g_{1},g_{2})\leq \eta-D$. Let $\alpha$ be a geodesic connecting $g_{1}$ and $g_2$. By the triangle inequality we have $d(\alpha,P)\geq D$ and so $$diam(\mathrm{proj}_{P}(\alpha))<a_0.$$ 
In particular $d(\pi_{P}(g_1),\pi_{P}(g_1))<a_0$. Thus, the conditions $d(g_{1},P)>\eta$ and $d(\pi_{P}(g_1),\pi_{P}(g_1))\geq a_0$ imply that $d(g_{1},g_{2})>\eta-D$, completing the proof, since $\eta$ tends to infinity.
\end{proof}


The following classical lemma is due to Kesten (see \cite{Kesten}). It holds as soon as $\Gamma$ is nonamenable.

\begin{lemma}[Kesten]\label{lemmakesten}
Denote by $\mu^{*n}$ the $n$th power of convolution of the measure $\mu$.
There exists $\alpha>0$ such that for every $g \in \Gamma$,
$$\mu^{*n}(g)\leq \mathrm{e}^{-\alpha n}.$$
\end{lemma}

\medskip
We can now prove Proposition~\ref{exponentialmoments}.
Let $z\in \mathbb{Z}^k$ and $j,j'\in \{1,...,N\}$.
If the first return to $N_{\eta}(P)$ starting at $(0,j)$ is at $(z,j')$,
there is a path $Z_0,...,Z_{n+1}$ such that $Z_0=(0,j)$, $Z_{n+1}=(z,j')$ and $Z_l\notin N_{\eta}(P)$ for $1\leq l \leq n$.
Since $d(Z_l,Z_{l+1})\leq r(\mu)$, where $r(\mu)$ only depends on $\mu$, if ${\eta}\geq 3r(\mu)$, then $Z_0,Z_{n+1}\notin N_{2{\eta}/3}(P)$ as soon as $n\geq 1$, which will hold if $\|z\|$ is large enough.
Moreover, any geodesic from $Z_l$ to $Z_{l+1}$ stays outside of $N_{{\eta}/3}(P)$, for $0\leq l \leq n$.

Define a path $\phi$ from $Z_0$ to $Z_{n+1}$ by gluing together geodesics from $Z_l$ to $Z_{l+1}$.
Then, the length of $\phi$ is smaller or equal to $nr(\mu)$.
The parabolic subgroup $P$ together with the word distance is quasi-isometric to its subgroup $\Z^k$ together with the euclidean distance.
In particular, the word distance between $0$ and $z$ is larger than $\Lambda \|z\|$, where $\Lambda$ only depends on the quasi-isometry parameters.

Using Lemma~\ref{lemmacontraction1}, we can choose consecutive points $y_1,...,y_l$ on $\phi$ which project on $P$ on points $\tilde{y}_1,...,\tilde{y}_l$ such that the distance between $\tilde{y}_i$ and $\tilde{y}_{i+1}$ is between $a_0$ and $2a_0$, for $1\leq i \leq l-1$ and such that $l\geq  \frac{\Lambda \|z\|}{2a_0}$.
(where $a_0$ is the constant in Lemma~\ref{lemmacontraction2}).
By gluing together a path from $0$ to $\tilde{y}_1$, paths from $\tilde{y}_i$ to $\tilde{y}_{i+1}$ and a path fom $\tilde{y}_{l}$ to $z$, we get a path from $0$ to $z$ inside $P$ whose length is thus larger than $\Lambda \|z\|$.
We deduce from Lemma~\ref{lemmacontraction2} that $d(y_j,y_{j+1})\geq \rho(\eta/3)$, so that the length of $\phi$ is larger or equal to $\frac{\Lambda}{2a_0}\rho(\eta/3)\|z\|$,
where $\rho({\eta}/3)$ tends to infinity when ${\eta}$ tends to infinity.

Fix $R_0\geq 0$.
Then, for large enough ${\eta}$, $n\geq R_0\|z\|$.
Summing over all paths from $\gamma=(0,j)$ to $\gamma_z=(z,j')$ that stay outside $N_{\eta}(P)$, we have
$$p_{j,j'}(0,z)\leq \sum_{n\geq R_0\|z\|}\mu^{*n}(\gamma^{-1}\gamma_z).$$
Lemma~\ref{lemmakesten} shows that
$$p_{j,j'}(0,z)\leq \sum_{n\geq R_0\|z\|}\mathrm{e}^{-\alpha n}\leq \mathrm{e}^{-\alpha R_0 \|z\|}\sum_{n\geq 0}\mathrm{e}^{-\alpha n}.$$
To prove Proposition~\ref{exponentialmoments}, it suffices to choose $R_0$ so that $R_0\alpha>M$.
\qed

\medskip

We can now complete the proof of Proposition~\ref{propconvergenceneighborhoodhorospheres}, using Corollary~\ref{markovsummary}.
Indeed, Lemma \ref{strongirreducibility}  shows that the induced chain on (arbitrary) bounded neighborhoods of $P$ is strongly irreducible, while Lemma~\ref{lemmasubMarkovneighborhoodhorosphere} shows that it is strictly sub-Markov (Condition~a) of Corollary~\ref{markovsummary}) and  Proposition~\ref{exponentialmoments} shows that it has sufficiently high exponential moment (Condition~b) of Corollary~\ref{markovsummary}).

Thus, Corollary ~\ref{markovsummary} implies that the Martin compactification of the induced chain on bounded neighborhoods of $P$ coincides with the geometric compactification of $P$, and together with Lemma \ref{sameGreen} this implies 
Proposition~\ref{propconvergenceneighborhoodhorospheres}.
\qed

\medskip

To prove Proposition~\ref{propconvergencehorospheres}, we now show that we can reduce to the case of a sequence that stays in a uniform neighborhood of the parabolic subgroup $P$. The proof is based on the following strategy.
Assume that the sequence $(g_n)$ leaves every bounded neighborhood of $P$ (but its projections to $P$ still converge to a point $\theta \in \partial P$). Proposition~\ref{propconvergenceneighborhoodhorospheres} applied to $(\pi_{P}g_n)$ guarantees that $(\pi_{P}g_n)$ converges to a point $\alpha$ in the Martin boundary.
We want to prove that the same is true for $g_n$. In other words, we want to prove that $K_{g_n}$ converges pointwise to a function $K_{\alpha}$.
Relative Ancona inequalities show that to go from a basepoint $o$ or from an arbitraty point $g$ to $g_n$, the random walk visits $\pi_P(g_n)$ with high probability.
Thus, $G(g,g_n)$ is close to $G(g,\pi_P(g_n))G(\pi_P(g_n),g_n)$ and $G(o,g_n)$ is close to $G(o,\pi_P(g_n))G(\pi_P(g_n),g_n)$, so that
$K_{g_n}(g)$ is close to $K_{\pi_P(g_n)}(g)$. Convergence for $K_{g_n}(g)$ then follows from convergence for $K_{\pi_P(g_n)}(g)$.

We now give a formal proof.
Let $z_{n}=\pi_{P}g_{n}$ be a projection point of $g_n$ to $P$.
By Lemma~8.2 of \cite{GGPY}, there is a (uniform, depending only on $\Gamma$) $\delta>0$ with $\lim \inf_{n\to \infty} \delta^{f}_{z_{n}}(x,g_{n})>\delta$ for all $x\in \Gamma$.
Let $\epsilon>0$.
Consider any $x\in \Gamma$.
By Theorem \ref{Ancona-Floyd}, there is a $\eta>0$ such that for large enough $n$,
\begin{equation}\label{equationMartinparabolic1}
G(x,g_{n};B^{c}_{\eta}(z_n))\leq \epsilon G(x,g_n)
\end{equation}
and
\begin{equation}\label{equationMartinparabolic2}
G(o,g_{n};B^{c}_{\eta}(z_n))\leq \epsilon G(o,g_n).
\end{equation}

Assume $\eta$ is also large enough to satisfy Proposition~\ref{exponentialmoments}.
Decomposing a path from $o$ to $g_n$ according to its last visit to $B_{\eta}(z_n)$, we can write
\begin{equation}\label{equationMartinparabolic3}
G(o, g_{n})=\sum_{u_{n}\in B_{\eta}(z_{n})}G(o,u_{n})G(u_{n},g_{n};B^{c}_{\eta}(z_{n}))+G(o,g_{n};B^{c}_{\eta}(z_{n}))
\end{equation}
and similarly,
\begin{equation}\label{equationMartinparabolic4}
G(x, g_{n})=\sum_{u_{n}\in B_{\eta}(z_{n})}G(x,u_{n})G(u_{n},g_{n};B^{c}_{\eta}(z_{n}))+G(x,g_{n};B^{c}_{\eta}(z_{n}))
\end{equation}

By Proposition \ref{propconvergenceneighborhoodhorospheres} we know that for any $u_{n}\in B(z_{n},\eta)$, $G(x,u_{n})/G(o,u_{n})$ converges to some $K_{\alpha}(x)$ where $\alpha\in \partial_{\mu}\Gamma$ is independent of $(u_n)$ and so, for large enough $n$, we have
\begin{equation}\label{equationMartinparabolic5}
(1-\epsilon)K_{\alpha}(x) \leq \frac{G(x,u_{n})}{G(o,u_{n})}\leq (1+\epsilon)K_{\alpha}(x).
\end{equation}

Combining~(\ref{equationMartinparabolic1}), (\ref{equationMartinparabolic4}) and~(\ref{equationMartinparabolic5}), we obtain for all large $n$ that
$$G(x,g_n)\leq \sum_{u_{n}\in B_{\eta}(z_{n})}(1+\epsilon)K_{\alpha}(x)G(o,u_{n})G(u_{n},g_{n};B^{c}_{\eta}(z_{n})) +\epsilon G(x,g_n),$$
so that
$$(1-\epsilon)G(x,g_n)\leq (1+\epsilon)K_{\alpha}(x)\sum_{u_{n}\in B_{\eta}(z_{n})}G(o,u_{n})G(u_{n},g_{n};B^{c}_{\eta}(z_{n}))$$
and then, using (\ref{equationMartinparabolic3}), $(1-\epsilon)G(x,g_n)\leq (1+\epsilon)K_{\alpha}(x)G(o,g_n)$.
Similarly, using (\ref{equationMartinparabolic2}), (\ref{equationMartinparabolic3}), (\ref{equationMartinparabolic4}) and~(\ref{equationMartinparabolic5}), we get a lower bound, so that for large enough $n$,
$$\frac{1-\epsilon}{1+\epsilon}K_{\alpha}(x) \leq \frac{G(x,g_{n})}{G(o,g_{n})}\leq \frac{1+\epsilon}{1-\epsilon}K_{\alpha}(x).$$

Since $\epsilon>0$ is arbitrary we get that $(g_{n})$ converges to $\alpha$ in the Martin boundary, completing the proof of Proposition \ref{propconvergencehorospheres}.
\qed

\medskip

We have shown that if $(g_n)$ is a sequence in $\Gamma$ such that either:
\begin{enumerate}
\item $(g_{n})$ converges to a conical point of the Bowditch boundary, or
\item For some maximal parabolic subgroup $P$, the projections $(\pi_{P} g_{n})$ converge to a point of $\partial P$,
\end{enumerate}
then $(g_n)$ converges in the Martin compactification.
As explained above, the same holds if, for some coset $gP$ of a parabolic group $P$, $(\pi_{gP} g_{n})$ converges in $g\partial P$.

To complete the proof that the Martin boundary is a $Z$-boundary we need to show the converse: namely that if $(g_n)$ converges to a point in the Martin boundary, then it satisfies either (1) or (2).

Suppose $(g_n)$ converges to a point in the Martin boundary. By  Proposition \ref{MartintoBowditch1}, $(g_n)$ converges to a point $\alpha$ in the Bowditch boundary. If $\alpha$ is conical, then (1) holds. Suppose now that $\alpha$ is parabolic, with stabilizer $P$. If (2) is not satisfied, there are subsequences $(h_{n})$ and $(h'_n)$ of $(g_n)$ with $(\pi_{P}h_n)$ and $(\pi_{P}h'_n)$ converging to different points of $\partial P$. By Proposition \ref{propconvergenceneighborhoodhorospheres}, $(\pi_{P}h_n)$ and $(\pi_{P}h'_n)$ converge to different points $\alpha$ and $\alpha'$ in the Martin boundary. Furthermore, by Proposition \ref{propconvergencehorospheres}, $(h_n)$ converges to the same point in the 
Martin boundary as $(\pi_{P}h_n)$ and $(h'_n)$ converges to the same point in the 
Martin boundary as $(\pi_{P}h'_n)$. Thus $(h_n)$ and $(h'_n)$ converge to different points of the Martin boundary, contradicting our assumption on $(g_n)$.

This proves that the Martin boundary is a $Z$-bondary, ending the proof of Theorem~\ref{maintheorem}.
Indeed, as stated above, the authors of \cite{GGPY} already proved that the identity of $\Gamma$ extends to an equivariant surjective map from the Martin compactification to the Bowditch compactification.
\qed

\bigskip
Using Theorem~\ref{maintheorem} and \cite[Theorem~1.3]{GGPY}, we recover the following corollary, which is interesting independently of the random walks context.
However, it could also be deduced from previous results of Gerasimov and Gerasimov-Potyagailo.

\begin{corollary}\label{coroZbdryFloyd}
Let $\Gamma$ be a finitely generated group, hyperbolic relative to a collection of infinite virtually abelian subgroups.
Then, there exists an equivariant and continuous surjective map from any $Z$-boundary to the Floyd boundary of $\Gamma$.
\end{corollary}

Since the Martin boundary does not depend on different peripheral  structures, but only on the random walk, our argument also implies the following.
\begin{corollary}
If $\Gamma$ is hyperbolic relative to two different collections of infinite virtually abelian subgroups, then two corresponding $Z$-boundaries, constructed for each relatively hyperbolic structure, are equivariantly homeomorphic, 
\end{corollary}

We conclude the discussion with few  questions related to the above argument.
Since $Z$-boundaries can be defined for more general relatively hyperbolic groups, it seems reasonable to ask the following question.

\medskip
1. \textit{Is Corollary~\ref{coroZbdryFloyd} true for more general relatively hyperbolic groups?} We conjecture that there are possible counter-examples.

\medskip

The following question, motivated by the proof of Proposition~\ref{propconvergencehorospheres}, also seems to be interesting.

\medskip

2. \textit{For which classes of groups $P$ does the following hold ? Assume that $\Gamma$ is hyperbolic relative to $P$ and that $(x_n)$ is a sequence of elements of $\Gamma$ leaving every compact. Then, if $y_n$ is the projection of $x_n$ onto $P$ (or actually its projection onto any horosphere based at $P$, see \cite{GePoCrelle} for the precise definition of a horosphere in this context), one has that $\frac{G(x,x_n)}{G(o,x_n)}\frac{G(x,y_n)}{G(o,y_n)}$ converges to 1}.

\section{Minimality}\label{Sectionminimality}

In this section we prove Theorem~\ref{minimal} from the introduction, namely the minimality of the Martin boundary.
We will use the following result of Dussaule \cite[Proposition~6.3]{Dussaule}.

\begin{proposition}\label{uniformseparation}
Let $p$ be a strongly irreducible transition kernel on $\Z^k\times \{1,...,N\}$ which is $\Z^k$-invariant and satisfies Assumptions~\ref{Assumption1} and \ref{Assumption2} of section~\ref{Sectionthickenedperipheralgroups}.
Let $\theta_0\neq \theta_1$ be two points in the Martin boundary $\partial \Z^k$.
There exists a neighborhood $\mathcal{U}$ of $\theta_1$ in $\partial \Z^d$ and a sequence $(x_n,k_n)$ of $\Z^d\times \{1,...,N\}$ such that
for every $\theta$ in $\mathcal{U}$, $K_{\theta}((x_n,k_n))$ tends to infinity, uniformly in $\theta$ and $K_{\theta_0}((x_n,k_n))$ converges to 0.
\end{proposition}

Applied to our situation, this gives the following.
\begin{corollary}\label{sepcor}
There exists an $\eta_0>0$ such that for $\eta>\eta_0$ the following holds.
For any distinct $\alpha_{0},\alpha_{1}\in \partial P$ there exists a neighborhood $\mathcal{U}$ of $\alpha_1$ in $\partial P$ and 
and a sequence $(g_n)$ of $N_{\eta}P$ such that $K_{\alpha}(g_n)$ tends to infinity, uniformly over $\alpha \in \mathcal{U}$ and $K_{\alpha_0}(g_n)$ converges to $0$.
\end{corollary}

We now prove the following.
\begin{theorem}
Let $\Gamma$ be hyperbolic relative to a collection of virtually abelian subgroups. Let $\mu$ be a probability measure on $\Gamma$ whose finite support generates $\Gamma$ as a semigroup. Then every point of the Martin boundary $\partial_{\mu}\Gamma$ corresponds to a minimal harmonic function.
\end{theorem}

\begin{proof}
By Theorem \ref{maintheorem}, $\partial_{\mu}\Gamma$ is a $Z$-boundary.
This means that there is a $\Gamma$-equivariant surjective map
$F:\partial_{\mu}\Gamma \to \partial_{B}\Gamma$ such that if $a\in \partial_{B}\Gamma$ is conical, $F^{-1}(a)$ is a single point and if $a\in \partial_{B}\Gamma$ is parabolic, 
$F^{-1}(a)=\partial P$ where $P$ is the stabilizer of $a$ and $\partial P$ denotes its $CAT(0)$ boundary.
Notice that $F:\partial_{\mu}\Gamma \to \partial_{B}\Gamma$ is the same map as the the map $\psi:\partial_{\mathcal{M}}\Gamma \rightarrow \partial_B\Gamma$ constructed in \cite{GGPY} (see the discussion after Theorem~1.3 there).

Let $\alpha_{0} \in \partial_{\mu}\Gamma$.
Then $K_{\alpha_0}$ is a positive harmonic function.
By the Choquet representation theorem, there exists a finite Borel measure $\nu_{0}=\nu^{\alpha_0}$ on the Martin boundary, with support contained in the minimal Martin boundary $\partial^{m}_{\mu}\Gamma$ such that for all $g\in \Gamma$
$$K_{\alpha_{0}}(g)=\int_{\partial^{m}_{\mu}\Gamma}K_{\alpha}(g)d\nu_0(\alpha).$$

To prove minimality of $\alpha_0$ it suffices to show that the support of $\nu_{0}$ consists of the single point $\alpha_0$. 
In Corollary 7.9 of \cite{GGPY} the authors prove the following result, which is based on relative Ancona inequalities (see Theorem~\ref{Ancona-Floyd}).

\begin{lemma} \label{bowditchsupport}
The support of $K_{\alpha_0}$ is contained in $F^{-1}(F(\alpha_0))$.
\end{lemma}

If $F(\alpha_0)$ is conical, then $F^{-1}(F(\alpha_0))$ is a single point. Lemma~\ref{bowditchsupport} then implies that the support of $\nu_0$ is a single point so that $\alpha_0$ is minimal.

On the other hand, if $\alpha_0$ is a parabolic point of the Bowditch boundary with stabilizer $P$, Theorem \ref{maintheorem} implies that $F^{-1}(F(\alpha_0))=\partial P$. Thus we know $\nu_0$ is supported on $\partial P \cap \partial^{m}_{\mu}\Gamma$.

Now, suppose $\alpha_1$ is a point of $\partial P$ distinct from $\alpha_0$. By Corollary~\ref{sepcor}, there exists a neighborhood $\mathcal{U}$ of $\alpha_1$ and a sequence $(g_n)$ such that $K_{\alpha}(g_n)$ tends to infinity uniformly over $\alpha \in \mathcal{U}$ and $K_{\alpha_0}(g_n)$ converges to 0.
Thus, for large enough $n$ and for all $\alpha \in U$, we have $K_{\alpha}(g_n)\geq 1$.
Then by definition,
$$K_{\alpha_0}(g_n)=\int_{\alpha \in \partial P} K_{\alpha}(g_n)d\nu_{0}(\alpha) \geq 
\int_{\alpha \in U} K_{\alpha}(g_n)d\nu_{0}(\alpha)\geq \nu_0(\mathcal{U}).$$

As $K_{\alpha_0}(g_n)\to 0$ as $n\to \infty$ it follows that $\nu_{0}(\mathcal{U})=0$ so that the support of $\nu_0$ does not contain $\alpha$. We conclude that the support of $\nu_0$ consists only of $\alpha_0$, so $K_{\alpha_0}$ must be minimal.
\end{proof}

\appendix
\section{Construction of a $Z$-boundary}
We now use results of Dahmani in \cite{Dahmani} to show that if $\Gamma$ is hyperbolic relative to a collection of virtually abelian subgroups, there is a geometric construction of a $Z$-boundary.
His construction uses Farb's definition of relatively hyperbolic groups and we first give a brief outline of it.

\subsection{Farb's definition of relative hyperbolicity}
Let $\Gamma$ be a finitely generated group and $\Omega$ a family of finitely generated subgroups of $\Gamma$.
If $C_S\Gamma$ is a Cayley graph of $\Gamma$ containing the Cayley graphs of every subgroup in $\Omega$,
define the coned-off Cayley graph by adding one vertex $v(gP)$ for every coset of the form $gP$, with $g\in \Gamma$ and $P\in \Omega$ and adding one edge $e(g,\gamma)$ between $v(gP)$ and $g\gamma$, for $P$ in $\Omega$ and $\gamma$ in $P$.
Following Farb in \cite{Farb}, the group $\Gamma$ is said to be weakly hyperbolic relative to the collection $\Omega$ if the coned-off Cayley graph defined above is hyperbolic (in the sense of Gromov).
This definition is invariant under quasi-isometry and thus does not depend on the choice of the Cayley graph of $\Gamma$.
In the following, for simplicity, we will assume that there is only one conjugacy class in $\Omega$, and in that case, if $P\in \Omega$, then $\Gamma$ is said to be hyperbolic relative to $P$.
We fix a Cayley graph $C_S\Gamma$ containing the Cayley graph of $P$.
The coned-off Cayley graph is quasi-isometric to the quotient space obtained by identifying every two elements lying in the same left coset of $P$.
Denote by $\hat{\Gamma}$ this quotient space
which is thus hyperbolic and denote by $\partial \hat{\Gamma}$ its Gromov boundary.
The following is a reformulation by Dahmani in \cite{Dahmani} of the bounded coset penetration property of Farb in \cite{Farb}.

\begin{definition}
Let $\alpha:[a,b]\rightarrow C_S\Gamma$ be a path parametrized by arc length and let $\hat{\alpha}$ be its image in $\hat{\Gamma}$.
Re-parameter $\hat{\alpha}$ to remove all loops of length 1 (that is remove subpaths that lie in a same coset), so that it is still parametrized by arc length.
Say that $\alpha$ is a relative geodesic if this new parametrization of $\hat{\alpha}$ is a geodesic.
Say that it is a $(\lambda,c)$-relative-quasi-geodesic if the new parametrization of $\hat{\alpha}$ is a $(\lambda,c)$-relative-quasi-geodesic,
that is,
$$\frac{1}{\lambda}|t-s|-c\leq d(\hat{\alpha}(t),\hat{\alpha}(s))\leq \lambda |t-s|+c.$$
\end{definition}

\begin{definition}
The pair $(\Gamma,P)$ satisfies the bounded coset penetration property (BCP for short) if for all $\lambda,c$, there exists a constant $r$ such that for every pair $(\alpha_1,\alpha_2)$ of $(\lambda,c)$-relative-quasi-geodesic without loop, starting at the same point and ending at the same point in $C_S\Gamma$, the following holds
\begin{enumerate}
\item if $\alpha_1$ travels more than $r$ in a coset, then $\alpha_2$ enters this coset,
\item if $\alpha_1$ and $\alpha_2$ enter the same coset, the two entering points and the two exiting points are $r$-close to each other.
\end{enumerate}
\end{definition}

Say that $\Gamma$ is hyperbolic relative to $P$ if it is weakly hyperbolic relative to $P$ and the pair $(\Gamma,P)$ satisfies the BCP property.
Similarly, one can define hyperbolicity relative to a family of subgroups $\Omega$.
This definition of relative hyperbolicity is equivalent to those we gave in Section~\ref{Sectionrelativelyhyperbolicgroups} (see for example Appendix~A in \cite{Dahmanithesis}, where Dahmani proves it is equivalent to the definition of Bowditch in \cite{Bowditch}).

\subsection{Dahmani's geometric boundary}
In \cite{Dahmani}, the author introduces a compactification for general relatively hyperbolic groups. This construction is a $Z$-boundary in our situation.
For simplicity, we again assume that there is only one conjugacy class in $\Omega$, that is $\Gamma$ is hyperbolic relative to a parabolic group $P$, but the construction holds for more than one conjugacy class (see Section~6.2 in \cite{Dahmani}).

We will need the following definition.
\begin{definition}
Consider a path $\alpha$ in the Cayley graph $C_S\Gamma$ and denote by $g_0$ the starting point of $\alpha$. The path $\alpha$ is called left-reduced if it immediately leaves the coset $g_0P$ and never returns to it.
\end{definition}

\begin{definition}
Assume that $P\cup \partial P$ is a compactification of $P$.
Say that finite sets fade at infinity if for all finite subset $F$ of $P$ and for all open cover $\mathcal{U}$ of $P\cup \partial P$, all translates of $F$ but finitely many are contained in an element of $\mathcal{U}$.
\end{definition}

Also recall that $\partial \hat{\Gamma}$ is the Gromov boundary of the quotient space $\hat{\Gamma}$ obtained by identifying every coset of $P$ to one point.
Now, choose a system of representatives $\widetilde{\Gamma/P}$ of $\Gamma/P$ and define the boundary $\partial \Gamma$ of $\Gamma$ as
$$\partial \Gamma:=\left (\underset{\tilde{\gamma}\in \widetilde{\Gamma/P}}{\bigsqcup}\tilde{\gamma}\partial P\right )\sqcup \partial \hat{\Gamma}.$$

Dahmani \cite[Theorem~3.1]{Dahmani} proved the following.
\begin{theorem}\label{Dahmaniboundary}
Assume that finite sets fade at infinity in $P\cup \partial P$. Then, the discrete topology on $\Gamma$ extends to a topology on $\Gamma \cup \partial \Gamma$ which makes it a metrizable compact space.
With this topology, $\Gamma$ is dense in $\Gamma \cup \partial \Gamma$.
Moreover, we have the following characterization of convergence to the boundary.
Let $(\gamma_n)$ be a sequence in $\Gamma$ and $\xi \in \partial \Gamma$.
\begin{itemize}
\item If $\xi\in \partial \hat{\Gamma}$, then $(\gamma_n)$ converges to $\xi$ if and only if its image $(\hat{\gamma}_n)$ in $\hat{\Gamma}$ converges to $\xi$ (in the Gromov sense).
\item If $\xi\in \tilde{\gamma}\partial P$, with $\tilde{\gamma}\in \widetilde{\Gamma/P}$, then $(\gamma_n)$ converges to $\xi$ if and only if there is a sequence $(x_n)$ of points in $P$ and a sequence of left-reduced relative-quasi-geodesics (with bounded parameters) from $\tilde{\gamma}x_n$ to $\gamma_n$, such that $x_n$ converges to a point in $\partial P$.
\end{itemize}
\end{theorem}

We can rephrase convergence in $\tilde{\gamma}\partial P$ as follows.
Assume that $\gamma\in \Gamma$, $x\in P$, and there is a left-reduced relative-quasi-geodesic (with bounded parameters) from $\tilde{\gamma}x$ to $\gamma$.
According to the BCP property, if $\alpha$ is another relative-quasi-geodesic with bounded parameters from $\tilde{\gamma}P$ to $\gamma$, then the exiting point of $\alpha$ is a uniformly bounded distance away from $\tilde{\gamma}x$.
In particular, if $\pi_P\gamma$ is a closest point projection of $\gamma$ on $P$, then $\pi_P\gamma$ is a bounded distance away from $\tilde{\gamma}x$.
Thus, a sequence $(\gamma_n)$ converges to a point $\xi\in \tilde{\gamma}\partial P$ if and only if the sequence of closest point projections of $\gamma_n$ on $\tilde{\gamma}P$ converges to $\xi \in \tilde{\gamma}\partial P$.

Now, recall that in our context, the parabolic subgroup $P$ is virtually abelian.
We have already defined a compactification of $P$.
Identifying $P$ with $\Z^k\times \{1,...,N\}$, as in Section~\ref{Sectioncompactification},
and identifying $\gamma \in P$ with $(z,j)\in \Z^k\times \{1,...,N\}$,
a sequence $(z_n,j_n)$ converges in $\partial P$ if $z_n$ tends to infinity and $\frac{z_n}{\|z_n\|}$ converges to a point $\theta$ in $\Ss^{k-1}$.
Formally, one can choose sets of the form $U_{n,m}(\theta)\times \{1,...,N\}$ to form a countable system of neighborhoods of a point $\theta$ in the boundary, where
$$U_{n,m}(\theta)=V_n(\theta)\cup \{z\in \mathbb{Z}^k, \|z\|\geq m, \frac{z}{\|z\|}\in V_n(\theta)\},$$ with $V_n(\theta)$ a neighborhood of $\theta$ in $\mathbb{S}^{k-1}$.

\begin{lemma}\label{lemmafinitesets}
With this topology, finite sets fade at infinity in $P\cup \partial P$.
\end{lemma}

\begin{proof}
First, we notice that $\mathbb{Z}^k\cup \partial \mathbb{Z}^k$ satisfies a fellow-traveler property.
Precisely, if $(z_n)$ and $(z'_n)$ are two sequences in $\mathbb{Z}^k$ such that $\|z_n-z'_n\|$ is bounded by a constant and $z_n$ converges to a point $\theta$  in $\partial \mathbb{Z}^k$, that is $\|z_n\|$ tends to infinity and $\frac{z_n}{\|z_n\|}$ converges to $\theta\in \mathbb{S}^{k-1}$, then $z'_n$ also converges to $\theta$ in $\partial \mathbb{Z}^k$. By induction, the same holds with a finite number of sequences, that is, if $(z^{(1)}_n),...,(z^{(j)}_n)$ are sequences in $\mathbb{Z}^k$ such that $\|z^{(j_1)}_n-z^{(j_2)}_n\|$ is bounded for every $j_1,j_2$ and such that one of them converges in $\partial \mathbb{Z}^k$, then they all converge to the same point.

Assume by contradiction that $F$ is a finite subset of $P$ and $\mathcal{U}$ is an open cover of $P\cup \partial P$ such that there are infinitely many translates of $F$ that are not contained in one of the open sets in $\mathcal{U}$.
Denotes these translates by $\gamma_n\cdot F$.
Let $f\in F$. The sequence $(\gamma_n\cdot f)$ tends to infinity.
Up to taking a sub-sequence, it converges in $\partial P$.
By the fellow-traveler property described above, for every $f'\in F$, $(\gamma_n\cdot f')$ converges to the same point.
Eventually, they all lie in a neighborhood of this limit point, which is a contradiction.
\end{proof}

We use this compactification of $P$ to construct $\partial \Gamma$ as above.
The left multiplication action on $\Gamma$ extends to an action on $\Gamma \cup \partial \Gamma$ in the obvious way.
In \cite{Dahmani}, the author claims (without a proof) that the compactification $\Gamma \cup \partial \Gamma$ projects on the Bowditch compactification
$\Gamma\cup \partial_B\Gamma$.
Indeed, we have the following.

\begin{lemma}\label{DahmanitoBowditch}
With these notations, the identity map on $\Gamma$ extends to an equivariant, continuous and surjective map
$$\Gamma \cup \partial \Gamma\rightarrow \Gamma \cup \partial_B \Gamma.$$
Moreover, a point $\xi\in \hat{\Gamma}$ is mapped to a conical point in $\partial_B \Gamma$ and a point $\xi \in \tilde{\gamma}\partial P$, for some $\tilde{\gamma} \in \widetilde{\Gamma/P}$ is mapped to a parabolic point in $\partial_B \Gamma$.
\end{lemma}

\begin{proof}
We first construct a map $\partial \Gamma \rightarrow \partial_B \Gamma$.
Consider the space $Y$ which consists on the space $X$ with disjoint horoballs removed at every parabolic point
and endow $Y$ with the length metric coming from the induced distance of $X$ on $Y$.
The map $\Gamma \rightarrow Y$ given by $\gamma\mapsto \gamma\cdot o$ (where $o$ is a fixed base point) is a quasi-isometry.
If a sequence $(\gamma_n)$ converges in $\partial \Gamma$, then its image in $X$ tends to infinity.
If it converges to $\partial \hat{\Gamma}$, then the Gromov product of $\gamma_n$ and $\gamma_m$ based at some point in $\hat{\Gamma}$ tends to infinity, when $n$ and $m$ tend to infinity.
Thus, geodesics in $\hat{\Gamma}$ from a fixed base point to $\gamma_n$ and to $\gamma_m$ travel together for arbitrary large time.
This provides geodesics in the Cayley graph $C_S\Gamma$ that travel together for arbitrary large time and spend arbitrary large time outside cosets of peripheral subgroups.
The images in $X$ are quasi-geodesics that travel together for arbitrary large time and spend arbitrary large time outside horoballs corresponding to parabolic points.
Thus, the image of $(\gamma_n)$ converges to a conical point.
Now, if $(\gamma_n)$ converges to $\tilde{\gamma}\partial P$, then its projection on $\tilde{\gamma}P$ converges to the same point and stays in the coset $\tilde{\gamma}P$ of $P$.
Thus, the projection $(\gamma'_n)$ of the image of $(\gamma_n)$ in $X$ on the corresponding horosphere $H$ converges to the corresponding parabolic point.
Denote this parabolic point by $\alpha$.
A geodesic from a fixed base point $o$ to $\gamma_n$ follows approximately a geodesic from $o$ to $\gamma'_n$, then from $\gamma'_n$ to $\gamma_n$ (see \cite[Proposition~3.4]{Maher}), so that the Gromov product of $\gamma_n$ and $\alpha$ based at $o$ converges to infinity.
Thus, $(\gamma_n)$ also converges to $\alpha$.
Thus, if $\gamma \in \partial \Gamma$, there is a uniquely defined $\tilde{\gamma}\in \partial_B \Gamma$ and we can consider the map $\gamma \in \partial \Gamma \mapsto \tilde{\gamma} \in \partial_B \Gamma$.

By construction, if $(g_n)\in \Gamma$ converges to $\gamma$, then $(g_n\cdot o)$ converges to $\tilde{\gamma}$ in the Bowditch boundary.
By density of $\Gamma$ in the Bowditch compactification, this map is surjective and by compactness of the Bowditch compactification and density of $\Gamma$ in $\Gamma\cup \partial \Gamma$, it is continuous (see the similar proof of Lemma~\ref{equivalentZbdry}).
By construction, it is equivariant and continuously extends the identity map $\Gamma \rightarrow \Gamma$.
\end{proof}

Theorem~\ref{Dahmaniboundary} together with Lemmas~\ref{lemmafinitesets} and~\ref{DahmanitoBowditch} prove that $\partial \Gamma$ is a $Z$-boundary.
As stated above, Dahmani's construction still holds when there is more than one conjugacy class of parabolic subgroups (see \cite[Section~6.2]{Dahmani}).

\bibliographystyle{plain}
\bibliography{rel_hyp_groups}
\end{document}